\documentclass[a4paper,12pt]{article}
\usepackage{geometry}

\usepackage{amsmath}
\usepackage{amsthm}
\usepackage{amssymb}
\usepackage{mathtools}
\usepackage{multirow}

\usepackage{hyperref}
\hypersetup{
    colorlinks=true,
    allcolors=blue, }

\usepackage{tikz}
\usetikzlibrary{decorations.pathreplacing, calligraphy, positioning, math, decorations.markings, arrows, shapes, calc}

\newcommand{\seqnum}[1]{\href{https://oeis.org/#1}{\rm \underline{#1}}}

\usepackage{setspace}
\newtheorem{theorem}{Theorem}
\newtheorem{corollary}{Corollary}

\allowdisplaybreaks 
\usepackage{setspace}

\makeatletter
 
  \@addtoreset{equation}{section}
\makeatother

\makeatletter
\def\tbcaption{\def\@captype{table}\caption}
\makeatother

\theoremstyle{remark}
\newtheorem{remark}[theorem]{Remark}

\newtheorem*{theorem*}{Theorem}

\DeclarePairedDelimiter\ceil{\lceil}{\rceil}
\DeclarePairedDelimiter\floor{\lfloor}{\rfloor}

\DeclarePairedDelimiterX{\cif}[1]{(}{)}{\delimsize(#1\delimsize)}

\begin{document}
\UseRawInputEncoding
\title{Horadam cubes}

\author{Luka Podrug\\
    Faculty of Civil Engineering, University of Zagreb\\
    Croatia, Zagreb 10000\\
  \text{luka.podrug@grad.unizg.hr}}

\date{}
\maketitle

\sloppy

\begin{abstract} We define and investigate a new three-parameter family of graphs that further generalizes the Fibonacci and metallic cubes. Namely, the number of vertices in this family of graphs satisfies Horadam recurrence, a linear recurrence of second order with constant coefficients.  It is shown that the new family preserves many appealing and useful properties of the Fibonacci and metallic cubes. In particular, we present recursive decomposition and decomposition into grids. Furthermore, we explore metric and enumerative properties such as the number of edges, distribution of degrees, and cube polynomials. We also investigate the existence of Hamiltonian paths and cycles.   
\end{abstract}



\section{Introduction}\label{intr}

Hypercubes are one of the most famous families of graphs. The set of vertices of the hypercube of dimension $n$, denoted by $Q_n$, contains all binary strings of length $n$, and two vertices are adjacent if the Hamming distance between them is one. Recall that the Hamming distance of two strings is defined as the number of positions on which two strings differ. Hypercubes have been studied, among other reasons, for their potential use as processor interconnection networks of low diameter. However, their appealing properties often come with a cost of inflexible restriction of the number of vertices in the network. To overcome this disadvantage, many families of graphs have been introduced, among them the Fibonacci and Lucas cubes \cite{hsu1, MunCipSalv}. Fibonacci cube of dimension $n$, denoted by $\Gamma_n$ has binary strings as vertices, but strings that contain two consecutive ones are omitted. The adjacency rules are the same as in hypercubes. Binary strings with that property are called {\em Fibonacci strings}. Of course, the vertices of Fibonacci cubes are counted by Fibonacci numbers. Fibonacci cubes attracted a lot of attention and were thoroughly investigated. Many of those results are presented in the excellent recent book {\em Fibonacci Cubes with Applications and Variations}  \cite{FibCubesBook}. A similar approach was used in the definition of the Lucas cubes, but except for omitting the strings with two consecutive ones, the first and last letter in each string can not be both one \cite{MunCipSalv}. Their set of vertices is counted by the Lucas numbers.  In the following years, many generalizations of the Fibonacci and Lucas cubes are considered \cite{hsu2, GenLucCub}, and mostly were obtained by some additional restrictions on the binary strings. In a recent paper \cite{Munarini}, the author chose a new approach. Besides $0$ and $1$ in the alphabet, the author added a new letter $2$. Two vertices are adjacent in one can be obtained from another by replacing $0$ with $1$ or by replacing the block $11$ with $22$, thus obtaining a family of graphs whose vertices are counted by the Pell numbers. So he called the new family the Pell graphs. A similar approach is again used in the definition of the metallic cubes \cite{metallic} and generalized Pell graphs \cite{GenPellGraphs}. The definition of generalized Pell graphs is similar to the definition of the Pell graphs but uses an extended alphabet. The metallic cube $\Pi^a_n$ is defined as a graph whose vertices are strings of length $n$ on the alphabet $\left\lbrace 0,1,\ldots,a\right\rbrace$ where letter $a$ can appear only after $0$. Two vertices are adjacent if they differ in only one position and only by one. Vertices in the metallic cube $\Pi^a_n$ are counted by the sequence $s^a_n$ that satisfies the recurrence $s^a_n=as^a_{n-1}+s^a_{n-2}$ with initial values $s^a_0=1$ and $s^a_1=a$. Generalized Pell graphs and metallic cubes turn out to be promising generalizations, both retaining many appealing properties of the hypercubes and Fibonacci cubes. For example, they are induced subgraphs of hypercubes and median graphs and they both admit the recursive decomposition. 

In this paper, we take a step further. We define and investigate the family of graphs whose number of vertices is given by the sequence that satisfies the recurrence $s^{a,b}_n=as^{a,b}_{n-1}+bs^{a,b}_{n-2}$, thus going beyond the Fibonacci and metallic cubes. Although such sequences were studied already at the beginning of the 20th century \cite{tagiuri}, and probably even earlier, they got their name after A. F. Horadam who popularized them in a series of papers in the early sixties \cite{Horadam0,  Horadam1, Horadam2}. So, we call the new family of graphs {\em the Horadam cubes}.

\section{Definitions}
Let $a,b$ be non-negative integers and let $\mathcal{S}^{a,b}$ denote the free monoid containing $a+b$ generators $\left\lbrace 0,1,2,\dots,a-1,0a,0(a+1),\dots,0(a+b-1)\right\rbrace$. The listed generators are sometimes also called \textit{primitive blocks}. By a \textit{string} we mean an element of monoid $\mathcal{S}^{a,b}$,  i.e., a word from the alphabet $\left\lbrace 0,1,2,\dots,a+b-1\right\rbrace$ with the property that letters $a, a+1,\dots, a+b-1$ can only appear immediately after $0$. Other letters can appear arbitrary. This means that every string from $\mathcal{S}^{a,b}$ can be decomposed in a unique way into primitive blocks. For strings $\alpha=x_1\cdots x_n$ and $\beta=y_1\cdots y_m$ we define their {\em concatenation} as $\alpha\beta=x_1\cdots x_n y_1\cdots y_m$. If $\mathcal{S}^{a,b}_n$ denotes the set of all elements from the monoid $\mathcal{S}^{a,b}$ with length $n$ we can easily obtain that the cardinal number $s^{a,b}_n=\left|\mathcal{S}^{a,b}_n\right|$, for $n\geq 2$, satisfies the Horadam's recursive relation \begin{equation}
s^{a,b}_{n}=as^{a,b}_{n-1}+bs^{a,b}_{n-2} \label{Horadam_recursion}
\end{equation}
with initial values $s^{a,b}_0=1$, $s^{a,b}_1=a$ and $s^{a,b}_{n}=0$ for $n<0$. Indeed, the string of length $n$ can end with an element from the set $\left\lbrace 0,1,\dots,a-1 \right\rbrace$  and the rest of the string can be formed in $s^{a,b}_{n-1}$ ways. That gives us $as^{a,b}_{n-1}$ such strings. If the string ends with a block $0l$ for some $a\leq l\leq a+b-1$ then the rest of the string can be formed in $s^{a,b}_{n-2}$ ways. Since there is $b$ ways to the choose block $0l$, we have $bs^{a,b}_{n-2}$ such strings. These two cases yield the recurrence (\ref{Horadam_recursion}).

In fact, the numbers $s^{a,b}_n$ satisfy the well-known identity
\begin{equation}\label{eq:horadam_number of vertices}
s^{a,b}_n=F_{n+1}(a,b)=\sum\limits_{k=0}^{\floor{n/2}} \binom{n-k}{k}a^{n-2k}b^k,
\end{equation}  
where $F_n(x,y)$ are the modified Fibonacci polynomials defined recursively for $n\geq 2$ as $F_{n}(x,y)=xF_{n-1}(x,y)+yF_{n-2}(x,y)$ with initial conditions $F_0(x,y)=1$ and $F_1(x,y)=x$.

Horadam's recursive relation generalizes many well-known sequences. Note that by setting $a=b=1$, one obtains the recursive relation for the (shifted) Fibonacci numbers (\seqnum{A000045} in \cite{oeis}) $$F_n=F_{n-1}+F_{n-2},\space n\geq 2.$$ 
By setting $a=2$ and $b=1$ we have a recursive relation for the (shifted) Pell numbers (\seqnum{A000129} in \cite{oeis}) $$P_n=2P_{n-1}+P_{n-2},\space n\geq 2,$$ 
and $a=1$ and $b=2$ yields (shifted) Jacobsthal numbers (\seqnum{A001045} in \cite{oeis}) $$J_n=J_{n-1}+2J_{n-2},\space n\geq 2.$$    

Starting from the recursive relation (\ref{Horadam_recursion}), one can readily obtain the generating function $S(x)=\sum_{n\geq 0}s^{a,b}_nx^n$ for the sequence $s^{a,b}_n$: \begin{align}\label{eq:Horadam_seq_generating_function}
S(x)=\dfrac{1}{1-ax-bx^2}.
\end{align}  

Now we move toward the definition of the Horadam cubes. Let $\Pi_n^{a,b}$ denote the graph such that $ V\left(\Pi_n^{a,b}\right)=\mathcal{S}^{a,b}_n$ and for any $v_1,v_2\in V\left(\Pi_n^{a,b}\right)$ we have  $v_1v_2\in E\left(\Pi_n^{a,b}\right)$  if and only if one vertex can be obtained from the other by replacing a single letter $k$ with $k+1$ for $0\leq k\leq a+b-2$. Alternatively, $v_1=x_1\cdots x_n$ and $v_2=y_1\cdots y_m$ are adjacent if and only if $\sum\limits_{k=1}^n|x_i-y_i|=1$. The graph $\Pi^{a,b}_0$ is a trivial graph, i.e., the graph that contains a single vertex. Note that for $b=1$, one obtains the metallic cubes, and for $a=1$ and $b=1$, the Fibonacci cubes. As an example, Figure \ref{fig: Jabotsthal cubes} shows Horadam cubes for $a=1$ and $b=2$, which can also be referred to as {\em Jacobsthal cubes}, since their vertices are counted by Jacobsthal numbers $J_n$. Another example for $a=3$ and $b=2$ is presented in Figure \ref{fig:Pi_n^{3,2}}.  

\begin{figure}[h!] \centering \begin{tikzpicture}[scale=0.6]
\tikzmath{\x1 = 0.35; \y1 =-0.05; \z1=90; \w1=0.2;   \xs=-50; \ys=0;
\x2 = \x1 + 1; \y2 =\y1 +3; } 
\small
\node [label={[label distance=\y1 cm]\z1: $000$},circle,fill=blue,draw=black,scale=\x1](A3) at (0,6) {};
\node [label={[label distance=\y1 cm]\z1: $001$},circle,fill=blue,draw=black,scale=\x1](A2) at (1.5,7.5) {};
\node [label={[label distance=\y1 cm]\z1: $002$},circle,fill=blue,draw=black,scale=\x1](A1) at (3,9) {};
\node [label={[label distance=\y1 cm]\z1: $010$},circle,fill=blue,draw=black,scale=\x1](A4) at (1.5,4.5) {};
\node [label={[label distance=\y1 cm]\z1: $020$},circle,fill=blue,draw=black,scale=\x1](A5) at (3,3) {};

\draw [line width=\w1 mm] (A1)--(A2)--(A3)--(A4)--(A5); 
\end{tikzpicture}\begin{tikzpicture}[scale=0.6]
\tikzmath{\x1 = 0.35; \y1 =-0.05; \z1=90; \w1=0.2;   \xs=-50; \ys=0;
\x2 = \x1 + 1; \y2 =\y1 +3; } 
\small
\node [label={[label distance=\y1 cm]\z1: $0000$},circle,fill=blue,draw=black,scale=\x1](A3) at (0,6) {};
\node [label={[label distance=\y1 cm]\z1: $0001$},circle,fill=blue,draw=black,scale=\x1](A2) at (1.5,7.5) {};
\node [label={[label distance=\y1 cm]\z1: $0002$},circle,fill=blue,draw=black,scale=\x1](A1) at (3,9) {};
\node [label={[label distance=\y1 cm]\z1: $0010$},circle,fill=blue,draw=black,scale=\x1](A4) at (1.5,4.5) {};
\node [label={[label distance=\y1 cm]\z1: $0020$},circle,fill=blue,draw=black,scale=\x1](A5) at (3,3) {};
\node [label={[label distance=\y1 cm]\z1: $0100$},circle,fill=blue,draw=black,scale=\x1](A6) at (-1.5,7.5) {};
\node [label={[label distance=\y1 cm]\z1: $0101$},circle,fill=blue,draw=black,scale=\x1](A7) at (0,9) {};
\node [label={[label distance=\y1 cm]\z1: $0102$},circle,fill=blue,draw=black,scale=\x1](A8) at (1.5,10.5) {};
\node [label={[label distance=\y1 cm]\z1: $0200$},circle,fill=blue,draw=black,scale=\x1](A9) at (-3,9) {};
\node [label={[label distance=\y1 cm]\z1: $0201$},circle,fill=blue,draw=black,scale=\x1](A10) at (-1.5,10.5) {};
\node [label={[label distance=\y1 cm]\z1: $0202$},circle,fill=blue,draw=black,scale=\x1](A11) at (0,12) {};
\draw [line width=\w1 mm] (A1)--(A2)--(A3)--(A4)--(A5) (A6)--(A7)--(A8) (A9)--(A10)--(A11)  (A3)--(A6)--(A9) (A2)--(A7)--(A10) (A1)--(A8)--(A11) ; 
\end{tikzpicture}
\begin{tikzpicture}[scale=0.6]
\tikzmath{\x1 = 0.35; \y1 =-0.05; \z1=90; \w1=0.2;   \xs=35; \ys=0;
\x2 = \x1 + 1; \y2 =\y1 +3; } 
\small

\node [label={[label distance=\y1 cm]\z1: $00000$},circle,fill=blue,draw=black,scale=\x1](A3) at (0,6) {};
\node [label={[label distance=\y1 cm]\z1: $00001$},circle,fill=blue,draw=black,scale=\x1](A2) at (1.5,7.5) {};
\node [label={[label distance=\y1 cm]\z1: $00002$},circle,fill=blue,draw=black,scale=\x1](A1) at (3,9) {};
\node [label={[label distance=\y1 cm]\z1: $00010$},circle,fill=blue,draw=black,scale=\x1](A4) at (1.5,4.5) {};
\node [label={[label distance=\y1 cm]\z1: $00020$},circle,fill=blue,draw=black,scale=\x1](A5) at (3,3) {};
\node [label={[label distance=\y1 cm]\z1: $00100$},circle,fill=blue,draw=black,scale=\x1](A6) at (-1.5,7.5) {};
\node [label={[label distance=\y1 cm]\z1: $00101$},circle,fill=blue,draw=black,scale=\x1](A7) at (0,9) {};
\node [label={[label distance=\y1 cm]\z1: $00102$},circle,fill=blue,draw=black,scale=\x1](A8) at (1.5,10.5) {};
\node [label={[label distance=\y1 cm]\z1: $00200$},circle,fill=blue,draw=black,scale=\x1](A9) at (-3,9) {};
\node [label={[label distance=\y1 cm]\z1: $00201$},circle,fill=blue,draw=black,scale=\x1](A10) at (-1.5,10.5) {};
\node [label={[label distance=\y1 cm]\z1: $00202$},circle,fill=blue,draw=black,scale=\x1](A11) at (0,12) {};

\node [label={[label distance=\y1 cm]\z1: $01000$},xshift=\xs,yshift=\ys,circle,fill=blue,draw=black,scale=\x1](A14) at (0,6) {};
\node [label={[label distance=\y1 cm]\z1: $01001$},xshift=\xs,yshift=\ys,circle,fill=blue,draw=black,scale=\x1](A13) at (1.5,7.5) {};
\node [label={[label distance=\y1 cm]\z1: $01002$},xshift=\xs,yshift=\ys,circle,fill=blue,draw=black,scale=\x1](A12) at (3,9) {};
\node [label={[label distance=\y1 cm]\z1: $01010$},xshift=\xs,yshift=\ys,circle,fill=blue,draw=black,scale=\x1](A15) at (1.5,4.5) {};
\node [label={[label distance=\y1 cm]\z1: $01020$},xshift=\xs,yshift=\ys,circle,fill=blue,draw=black,scale=\x1](A16) at (3,3) {};

\node [label={[label distance=\y1 cm]\z1: $02000$},xshift=2*\xs,yshift=2*\ys,circle,fill=blue,draw=black,scale=\x1](A19) at (0,6) {};
\node [label={[label distance=\y1 cm]\z1: $02001$},xshift=2*\xs,yshift=2*\ys,circle,fill=blue,draw=black,scale=\x1](A18) at (1.5,7.5) {};
\node [label={[label distance=\y1 cm]\z1: $02002$},xshift=2*\xs,yshift=2*\ys,circle,fill=blue,draw=black,scale=\x1](A17) at (3,9) {};
\node [label={[label distance=\y1 cm]\z1: $02010$},xshift=2*\xs,yshift=2*\ys,circle,fill=blue,draw=black,scale=\x1](A20) at (1.5,4.5) {};
\node [label={[label distance=\y1 cm]\z1: $02020$},xshift=2*\xs,yshift=2*\ys,circle,fill=blue,draw=black,scale=\x1](A21) at (3,3) {};
\draw [line width=\w1 mm] (A1)--(A2)--(A3)--(A4)--(A5) (A6)--(A7)--(A8) (A9)--(A10)--(A11)  (A3)--(A6)--(A9) (A2)--(A7)--(A10) (A1)--(A8)--(A11) (A12)--(A13)--(A14)--(A15)--(A16) (A17)--(A18)--(A19)--(A20)--(A21)  (A1)--(A12)--(A17) (A2)--(A13)--(A18) (A3)--(A14)--(A19) (A4)--(A15)--(A20) (A5)--(A16)--(A21); 
\end{tikzpicture}
\caption{The Jacobsthal cube for $n=3,4,5$.} \label{fig: Jabotsthal cubes}
\end{figure}
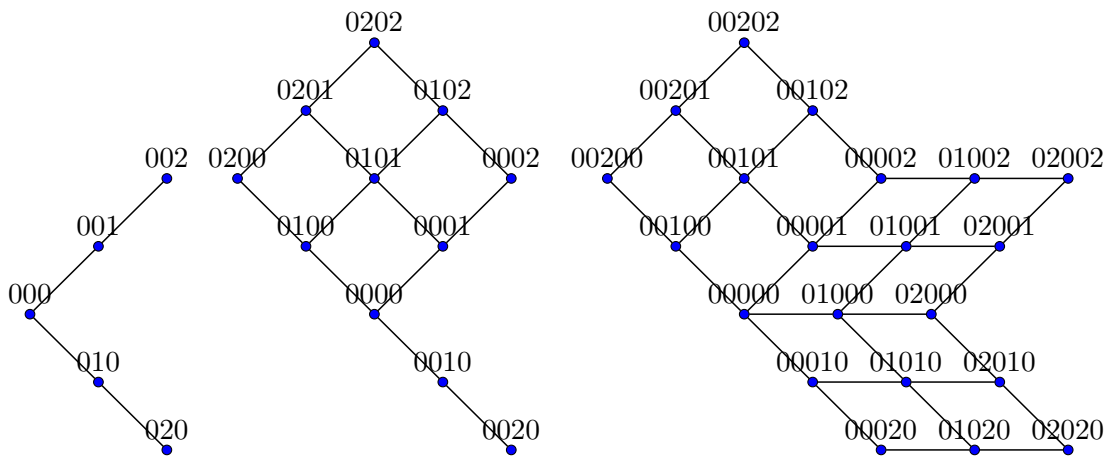

\begin{figure}[h!] \centering \begin{tikzpicture}[scale=0.6]
\tikzmath{\x1 = 0.35; \y1 =-0.05; \z1=180; \w1=0.2; \xs=-8; \ys=0; \yss=-5;
\x2 = \x1 + 1; \y2 =\y1 +3; } 
\small
\node [label={[label distance=\y1 cm]\z1: $0$},circle,fill=blue,draw=black,scale=\x1](A1) at (0,4) {};
\node [label={[label distance=\y1 cm]\z1: $1$},circle,fill=blue,draw=black,scale=\x1](A4) at (-1.5,2.5) {};
\node [label={[label distance=\y1 cm]\z1: $2$},circle,fill=blue,draw=black,scale=\x1](A7) at (-3,1) {};

\draw [line width=\w1 mm] (A1)--(A4)--(A7);

\end{tikzpicture} \centering \begin{tikzpicture}[scale=0.6]
\tikzmath{\x1 = 0.35; \y1 =-0.05; \z1=180; \w1=0.2; \xs=-8; \ys=0; \yss=-5;
\x2 = \x1 + 1; \y2 =\y1 +3; } 
\small
\node [label={[label distance=\y1 cm]\z1: $00$},circle,fill=blue,draw=black,scale=\x1](A1) at (0,4) {};
\node [label={[label distance=\y1 cm]\z1: $10$},circle,fill=blue,draw=black,scale=\x1](A2) at (0,5) {};
\node [label={[label distance=\y1 cm]\z1: $20$},circle,fill=blue,draw=black,scale=\x1](A3) at (0,6) {};
\node [label={[label distance=\y1 cm]\z1: $01$},circle,fill=blue,draw=black,scale=\x1](A4) at (-1.5,2.5) {};
\node [label={[label distance=\y1 cm]\z1: $11$},circle,fill=blue,draw=black,scale=\x1](A5) at (-1.5,3.5) {};
\node [label={[label distance=\y1 cm]\z1: $21$},circle,fill=blue,draw=black,scale=\x1](A6) at (-1.5,4.5) {};
\node [label={[label distance=\y1 cm]\z1: $02$},circle,fill=blue,draw=black,scale=\x1](A7) at (-3,1) {};
\node [label={[label distance=\y1 cm]\z1: $12$},circle,fill=blue,draw=black,scale=\x1](A8) at (-3,2) {};
\node [label={[label distance=\y1 cm]\z1: $22$},circle,fill=blue,draw=black,scale=\x1](A9) at (-3,3) {};
\node [label={[label distance=\y1 cm]\z1: $03$},circle,fill=blue,draw=black,scale=\x1](A28) at (-4.5,-0.5) {};
\node [label={[label distance=\y1 cm]\z1: $04$},circle,fill=blue,draw=black,scale=\x1](A37) at (-6,-2) {};

\draw [line width=\w1 mm] (A1)--(A2)--(A3) (A4)--(A5)--(A6) (A7)--(A8)--(A9)  (A1)--(A4)--(A7)--(A28)--(A37) (A2)--(A5)--(A8) (A3)--(A6)--(A9);

\end{tikzpicture} \begin{tikzpicture}[scale=0.6]
\tikzmath{\x1 = 0.35; \y1 =-0.05; \z1=180; \w1=0.2; \xs=-8; \ys=0; \yss=-5;
\x2 = \x1 + 1; \y2 =\y1 +3; } 
\small
\node [label={[label distance=\y1 cm]\z1: $000$},circle,fill=blue,draw=black,scale=\x1](A1) at (0,4) {};
\node [label={[label distance=\y1 cm]\z1: $010$},circle,fill=blue,draw=black,scale=\x1](A2) at (0,5) {};
\node [label={[label distance=\y1 cm]\z1: $020$},circle,fill=blue,draw=black,scale=\x1](A3) at (0,6) {};
\node [label={[label distance=\y1 cm]\z1: $001$},circle,fill=blue,draw=black,scale=\x1](A4) at (-1.5,2.5) {};
\node [label={[label distance=\y1 cm]\z1: $011$},circle,fill=blue,draw=black,scale=\x1](A5) at (-1.5,3.5) {};
\node [label={[label distance=\y1 cm]\z1: $021$},circle,fill=blue,draw=black,scale=\x1](A6) at (-1.5,4.5) {};
\node [label={[label distance=\y1 cm]\z1: $002$},circle,fill=blue,draw=black,scale=\x1](A7) at (-3,1) {};
\node [label={[label distance=\y1 cm]\z1: $012$},circle,fill=blue,draw=black,scale=\x1](A8) at (-3,2) {};
\node [label={[label distance=\y1 cm]\z1: $022$},circle,fill=blue,draw=black,scale=\x1](A9) at (-3,3) {};
\node [label={[label distance=\y1 cm]\z1: $100$},circle,fill=blue,draw=black,scale=\x1](A10) at (1.5,2.5) {};
\node [label={[label distance=\y1 cm]\z1: $110$},circle,fill=blue,draw=black,scale=\x1](A11) at (1.5,3.5) {};
\node [label={[label distance=\y1 cm]\z1: $120$},circle,fill=blue,draw=black,scale=\x1](A12) at (1.5,4.5) {};
\node [label={[label distance=\y1 cm]\z1: $101$},circle,fill=blue,draw=black,scale=\x1](A13) at (0,1) {};
\node [label={[label distance=\y1 cm]\z1: $111$},circle,fill=blue,draw=black,scale=\x1](A14) at (0,2) {};
\node [label={[label distance=\y1 cm]\z1: $121$},circle,fill=blue,draw=black,scale=\x1](A15) at (0,3) {};
\node [label={[label distance=\y1 cm]\z1: $102$},circle,fill=blue,draw=black,scale=\x1](A16) at (-1.5,-0.5) {};
\node [label={[label distance=\y1 cm]\z1: $112$},circle,fill=blue,draw=black,scale=\x1](A17) at (-1.5,0.5) {};
\node [label={[label distance=\y1 cm]\z1: $122$},circle,fill=blue,draw=black,scale=\x1](A18) at (-1.5,1.5) {};
\node [label={[label distance=\y1 cm]\z1: $200$},circle,fill=blue,draw=black,scale=\x1](A19) at (3,1) {};
\node [label={[label distance=\y1 cm]\z1: $210$},circle,fill=blue,draw=black,scale=\x1](A20) at (3,2) {};
\node [label={[label distance=\y1 cm]\z1: $220$},circle,fill=blue,draw=black,scale=\x1](A21) at (3,3) {};
\node [label={[label distance=\y1 cm]\z1: $201$},circle,fill=blue,draw=black,scale=\x1](A22) at (1.5,-0.5) {};
\node [label={[label distance=\y1 cm]\z1: $211$},circle,fill=blue,draw=black,scale=\x1](A23) at (1.5,0.5) {};
\node [label={[label distance=\y1 cm]\z1: $221$},circle,fill=blue,draw=black,scale=\x1](A24) at (1.5,1.5) {};
\node [label={[label distance=\y1 cm]\z1: $202$},circle,fill=blue,draw=black,scale=\x1](A25) at (0,-2) {};
\node [label={[label distance=\y1 cm]\z1: $212$},circle,fill=blue,draw=black,scale=\x1](A26) at (0,-1) {};
\node [label={[label distance=\y1 cm]\z1: $222$},circle,fill=blue,draw=black,scale=\x1](A27) at (0,0) {};
\node [label={[label distance=\y1 cm]\z1: $003$},circle,fill=blue,draw=black,scale=\x1](A28) at (-4.5,-0.5) {};
\node [label={[label distance=\y1 cm]\z1: $103$},circle,fill=blue,draw=black,scale=\x1](A29) at (-3,-2) {};
\node [label={[label distance=\y1 cm]\z1: $203$},circle,fill=blue,draw=black,scale=\x1](A30) at (-1.5,-3.5) {};
\node [label={[label distance=\y1 cm]\z1: $030$},circle,fill=blue,draw=black,scale=\x1](A31) at (-1.5,7.5) {};
\node [label={[label distance=\y1 cm]\z1: $031$},circle,fill=blue,draw=black,scale=\x1](A32) at (-3,6) {};
\node [label={[label distance=\y1 cm]\z1: $032$},circle,fill=blue,draw=black,scale=\x1](A33) at (-4.5,4.5) {};
\node [label={[label distance=\y1 cm]\z1: $040$},circle,fill=blue,draw=black,scale=\x1](A34) at (-3,9) {};
\node [label={[label distance=\y1 cm]\z1: $041$},circle,fill=blue,draw=black,scale=\x1](A35) at (-4.5,7.5) {};
\node [label={[label distance=\y1 cm]\z1: $042$},circle,fill=blue,draw=black,scale=\x1](A36) at (-6,6) {};
\node [label={[label distance=\y1 cm]\z1: $004$},circle,fill=blue,draw=black,scale=\x1](A37) at (-6,-2) {};
\node [label={[label distance=\y1 cm]\z1: $104$},circle,fill=blue,draw=black,scale=\x1](A38) at (-4.5,-3.5) {};
\node [label={[label distance=\y1 cm]\z1: $204$},circle,fill=blue,draw=black,scale=\x1](A39) at (-3,-5) {};

\draw [line width=\w1 mm] (A1)--(A2)--(A3) (A4)--(A5)--(A6) (A7)--(A8)--(A9) (A10)--(A11)--(A12) (A13)--(A14)--(A15) (A16)--(A17)--(A18) (A19)--(A20)--(A21) (A22)--(A23)--(A24) (A25)--(A26)--(A27) (A1)--(A4)--(A7)--(A28) (A2)--(A5)--(A8) (A3)--(A6)--(A9) (A10)--(A13)--(A16)--(A29) (A11)--(A14)--(A17) (A12)--(A15)--(A18) (A19)--(A22)--(A25)--(A30) (A20)--(A23)--(A26) (A21)--(A24)--(A27) (A31)--(A32)--(A33) (A34)--(A35)--(A36) (A28)--(A37) (A29)--(A38) (A30)--(A39) (A1)--(A10)--(A19) (A2)--(A11)--(A20) (A31)--(A3)--(A12)--(A21) (A4)--(A13)--(A22) (A5)--(A14)--(A23) (A32)--(A6)--(A15)--(A24) (A7)--(A16)--(A25) (A8)--(A17)--(A26) (A33)--(A9)--(A18)--(A27)  (A28)--(A29)--(A30) (A37)--(A38)--(A39) (A31)--(A34) (A32)--(A35) (A33)--(A36); 


\end{tikzpicture}
\caption{The Horadam cube $\Pi^{3,2}_n$ for $n=1,2,3$.} \label{fig:Pi_n^{3,2}}
\end{figure}
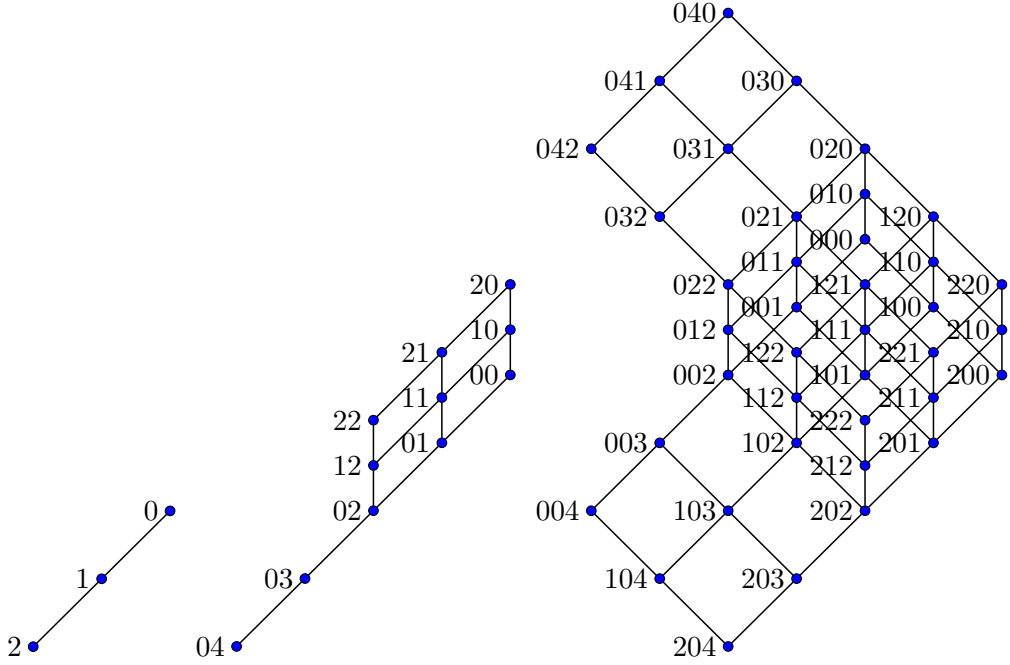

\section{Canonical decomposition and bipartivity}
The canonical decomposition of the Fibonacci cubes and metallic cubes \cite{metallic, MunSalvi} naturally extends to the Horadam cubes. Similar to those cases, the set of vertices $\mathcal{S}_n^{a,b}$ can be arranged into disjoint sets based on the starting letter. Namely, there are $b$ sets containing the vertices starting with block $0a, 0(a+1),\ldots, 0(a+b-2)$ and $0(a+b-1)$, and the remaining $a$ sets containing the vertices that start with $0, 1,\dots, a-2$ and $a-1$, respectively. Note that in the latter case, we only consider the vertices that start with $0$, but not with the block $0l$ for some $a\leq l\leq b$.

Assuming $\alpha\in \mathcal{S}^{a,b}_{n-1}$ and $\beta\in \mathcal{S}^{a,b}_{n-2}$,
the vertices $0\alpha$, $1\alpha$, $\dots$, $(a-1)\alpha$ generate $a$ copies
of $\Pi^{a,b}_{n-1}$ and vertices $0a\beta, 0(a+1)\beta,\ldots 0(a+b-1)\beta$ generate $b$ copies of $\Pi^{a,b}_{n-2}$. A copy of the subgraph $\Pi^{a,b}_{n-1}$ induced by vertices in $\Pi^{a,b}_n$ starting with the letter $k, 0\leq k\leq a-1$, is denoted by $k\Pi^{a,b}_{n-1}$, and a copy of the subgraph $\Pi^{a,b}_{n-2}$ induced by vertices in $\Pi^{a,b}_n$ starting with the block $0l, a\leq l\leq a+b-1$, is denoted by $0l\Pi^{a,b}_{n-2}$. This notation is useful if it is important to distinguish those copies. For example, $0\Pi^{a,b}_{n-1}$ denotes the induced subgraph of $\Pi^{a,b}_{n}$ isomorphic to $\Pi^{a,b}_{n-1}$ induced by vertices in $\Pi^{a,b}_n$ starting with $0$.

Before we state the theorem, it is useful to remind the reader that for two graphs $G$ and $H$, their Cartesian product is the graph denoted by $G\square H$ 
with $V(G\square H)=V(G)\times V(H)$ and $(u_1,v_1)(u_2,v_2)\in E(G\square H)$ if $u_1=u_2$ and $v_1$ and $v_2$ are adjacent in $H$, or $v_1=v_2$ and $u_1$ and $u_2$ are adjacent in $G$ \cite{Handbook_of_Product_Graphs}. 

\begin{theorem}\label{tm:candec_hor} For $n\geq 2$, the Horadam cube $\Pi^{a,b}_n$ admits the decomposition
\begin{equation*}
\Pi^{a,b}_n=\underbrace{\Pi^{a,b}_{n-1}\oplus\cdots\oplus\Pi^{a,b}_{n-1}}_{a \textup{ copies}} \oplus \underbrace{\Pi^{a,b}_{n-2}\oplus\cdots\oplus\Pi^{a,b}_{n-2}}_{b \textup{ copies}}=P_a\square\Pi^{a,b}_{n-1} \oplus P_b \square\Pi^{a,b}_{n-2}, 
\end{equation*} 
where $P_a$ and $P_b$ denote the path graphs of lengths $a$ and $b$, respectively. \end{theorem} 
The Horadam cube $\Pi^{a,b}_n$, besides the edges in $P_a\square\Pi^{a,b}_{n-1}$ and $P_b \square\Pi^{a,b}_{n-2}$ also contains the edges connecting the subgraphs $0a\Pi^{a,b}_{n-2}$ and  $0(a-1)\Pi^{a,b}_{n-2}\subset 0\Pi^{a,b}_{n-1}$. The decomposition in Theorem \ref{tm:candec_hor} is called the {\em canonical decomposition} of $\Pi^{a,b}_n$. 

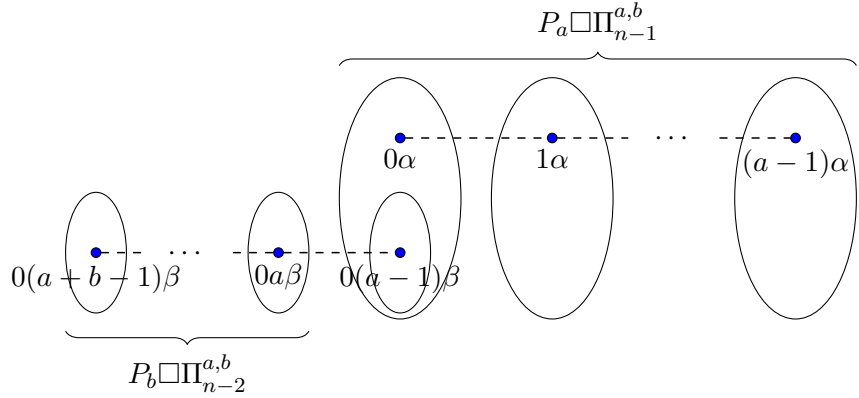
\begin{figure}\centering
\newcommand{\boundellipse}[3]
{(#1) ellipse (#2 and #3)
}

\begin{tikzpicture}[scale=0.8]\tikzmath{\x1 = 0.35; \y1 =-0.05; \z1=-90; \w1=0.2; \xs=-8; \ys=0; \yss=-5;
\x2 = \x1 + 1; \y2 =\y1 +3; } \small
\draw \boundellipse{-2.5,-0.9}{0.5}{1};
\draw \boundellipse{0.5,-0.9}{0.5}{1};
\draw \boundellipse{2.5,-0.9}{0.5}{1};
\draw \boundellipse{2.5,0}{1}{2};
\draw \boundellipse{5,0}{1}{2};
\draw \boundellipse{9,0}{1}{2};
\node [label={[label distance=\y1 cm]\z1: $0\alpha$},circle,fill=blue,draw=black,scale=\x1](A2) at (2.5,1) {};
\node [label={[label distance=\y1 cm]\z1: $1\alpha$},circle,fill=blue,draw=black,scale=\x1](A3) at (5,1) {};
\node [label={[label distance=\y1 cm]\z1: $(a-1)\alpha$},circle,fill=blue,draw=black,scale=\x1](A4) at (9,1) {};
\node [label={[label distance=-0.35 cm]\z1: $\cdots$}](A5) at (7,1) {};
\node [label={[label distance=-0.35 cm]\z1: $\cdots$}](A5) at (-1,-0.9) {};

\draw [line width=\w1 mm,dashed] (A2)--(A3)--(6.25,1)  (7.75,1)--(A4);

\node [label={[label distance=\y1 cm]\z1: $0a\beta$},circle,fill=blue,draw=black,scale=\x1](A8) at (0.5,-0.9) {};

\node [label={[label distance=\y1 cm]\z1: $0(a-1)\beta$},circle,fill=blue,draw=black,scale=\x1](A6) at (2.5,-0.9) {};
\node [label={[label distance=\y1 cm]\z1: $0(a+b-1)\beta$},circle,fill=blue,draw=black,scale=\x1](A10) at (-2.5,-0.9) {};

\draw [line width=\w1 mm,dashed] (-0.25,-0.9)--(A8)--(A6) (A10)--(-1.75,-0.9); 

\draw [black, left=3pt,
    decorate, 
    decoration = {brace,amplitude=5pt}](1.5,2.2)--(10,2.2) node[pos=0.5,above=5pt,black]{$P_a\square\Pi^{a,b}_{n-1}$};
    
\draw [black, left=3pt,
    decorate, 
    decoration = {brace,amplitude=5pt}](1,-2.2)--(-3,-2.2) node[pos=0.5,below=5pt,black]{$P_b\square\Pi^{a,b}_{n-2}$};
\end{tikzpicture}\caption{Canonical decomposition of $\Pi^{a,b}_n=P_a\square\Pi^{a,b}_{n-1}\oplus P_b\square \Pi^{a,b}_{n-2}$.}\label{fig_can_dec_diagram_hor}
\end{figure}

In Figure \ref{fig_can_dec_diagram_hor} we present a schematic representation of
the described canonical decomposition and Figure \ref{fig:canon dec and coloring of Pi_4^{2,2}} shows the canonical decomposition of the Horadam cube $\Pi_4^{2,2}$. 

A map $\chi:V(G)\to\left\lbrace 0,1 \right\rbrace$ is a {\em proper
$2$-coloring} of a graph $G$ if $\chi(v_1)\neq \chi(v_2)$ for every two
adjacent vertices $v_1,v_2\in V(G)$. A graph $G$ is {\em bipartite} if its
set of vertices $V(G)$ can be decomposed into two disjoint subsets such that
no two vertices of the same subset share an edge. Equivalently, a graph $G$
is bipartite if and only if it admits a proper $2$-coloring.

\begin{theorem}
All Horadam cubes are bipartite.
\end{theorem} \begin{proof}
Consider the map $\chi:\Pi^{a,b}_{n}\to \left\lbrace 0,1 \right\rbrace$ defined as follows: for $v=\alpha_1\cdots\alpha_n\in V(\Pi^{a,b}_{n})$, we define $$\chi(v)=\sum \alpha_i \space \mod 2.$$
Since any two vertices are adjacent if and only if they differ in only one position and by one, $\chi$ maps all neighboring vertices into different numbers. Thus we obtained a proper $2$-coloring for $\Pi^{a,b}_{n}$.  \end{proof}
Figure \ref{fig:canon dec and coloring of Pi_4^{2,2}}  shows a proper $2$-coloring of
the metallic cube $\Pi^{2,2}_4$.

\begin{figure}[h!] \centering \begin{tikzpicture}[scale=1]
\tikzmath{\x1 = 0.5; \y1 =-0.05; \z1=180; \w1=0.2; \xs=-125; \ys=-30; \yss=-5;
\x2 = \x1 + 1; \y2 =\y1 +3; } 
\scriptsize

\node [label={[label distance=\y1 cm]\z1: $0000$},fill=blue,draw=black,scale=\x1](A1) at (1,2) {};
\node [label={[label distance=\y1 cm]\z1: $0001$},circle,fill=red,draw=black,scale=\x1](A2) at (0,1) {};
\node [label={[label distance=\y1 cm]\z1: $0101$},fill=blue,draw=black,scale=\x1](A3) at (1,0) {};
\node [label={[label distance=\y1 cm]\z1: $0100$},circle,fill=red,draw=black,scale=\x1](A4) at (2,1) {};
\node [label={[label distance=\y1 cm]\z1: $0010$},circle,fill=red,draw=black,scale=\x1](A5) at (1,3) {};
\node [label={[label distance=\y1 cm]\z1: $0011$},fill=blue,draw=black,scale=\x1](A6) at (0,2) {};
\node [label={[label distance=\y1 cm]\z1: $0111$},circle,fill=red,draw=black,scale=\x1](A7) at (1,1) {};
\node [label={[label distance=\y1 cm]\z1: $0110$},fill=blue,draw=black,scale=\x1](A8) at (2,2) {};
\node [label={[label distance=\y1 cm]\z1: $0002$},fill=blue,draw=black,scale=\x1](A9) at (-0.75,0.25) {};
\node [label={[label distance=\y1 cm]\z1: $0003$},circle,fill=red,draw=black,scale=\x1](A10) at (-1.5,-0.5) {};
\node [label={[label distance=\y1 cm]\z1: $0102$},circle,fill=red,draw=black,scale=\x1](A11) at (0.25,-0.75) {};
\node [label={[label distance=\y1 cm]\z1: $0103$},fill=blue,draw=black,scale=\x1](A12) at (-0.5,-1.5) {};
\node [label={[label distance=\y1 cm]\z1: $0021$},circle,fill=red,draw=black,scale=\x1](A13) at (-0.75,2.75) {};
\node [label={[label distance=\y1 cm]\z1: $0031$},fill=blue,draw=black,scale=\x1](A14) at (-1.5,3.5) {};
\node [label={[label distance=\y1 cm]\z1: $0020$},fill=blue,draw=black,scale=\x1](A15) at (0.25,3.75) {};
\node [label={[label distance=\y1 cm]\z1: $0030$},circle,fill=red,draw=black,scale=\x1](A16) at (-0.5,4.5) {};

\node [label={[label distance=\y1 cm]\z1: $1000$},xshift=\xs,yshift=\ys,circle,fill=red,draw=black,scale=\x1](A17) at (1,2) {};
\node [label={[label distance=\y1 cm]\z1: $1001$},xshift=\xs,yshift=\ys,fill=blue,draw=black,scale=\x1](A18) at (0,1) {};
\node [label={[label distance=\y1 cm]\z1: $1101$},xshift=\xs,yshift=\ys,circle,fill=red,draw=black,scale=\x1](A19) at (1,0) {};
\node [label={[label distance=\y1 cm]\z1: $1100$},xshift=\xs,yshift=\ys,fill=blue,draw=black,scale=\x1](A20) at (2,1) {};
\node [label={[label distance=\y1 cm]\z1: $1010$},xshift=\xs,yshift=\ys,fill=blue,draw=black,scale=\x1](A21) at (1,3) {};
\node [label={[label distance=\y1 cm]\z1: $1011$},xshift=\xs,yshift=\ys,circle,fill=red,draw=black,scale=\x1](A22) at (0,2) {};
\node [label={[label distance=\y1 cm]\z1: $1111$},xshift=\xs,yshift=\ys,fill=blue,draw=black,scale=\x1](A23) at (1,1) {};
\node [label={[label distance=\y1 cm]\z1: $1110$},xshift=\xs,yshift=\ys,circle,fill=red,draw=black,scale=\x1](A24) at (2,2) {};
\node [label={[label distance=\y1 cm]\z1: $1002$},xshift=\xs,yshift=\ys,circle,fill=red,draw=black,scale=\x1](A25) at (-0.75,0.25) {};
\node [label={[label distance=\y1 cm]\z1: $1003$},xshift=\xs,yshift=\ys,fill=blue,draw=black,scale=\x1](A26) at (-1.5,-0.5) {};
\node [label={[label distance=\y1 cm]\z1: $1102$},xshift=\xs,yshift=\ys,fill=blue,draw=black,scale=\x1](A27) at (0.25,-0.75) {};
\node [label={[label distance=\y1 cm]\z1: $1103$},xshift=\xs,yshift=\ys,circle,fill=red,draw=black,scale=\x1](A28) at (-0.5,-1.5) {};
\node [label={[label distance=\y1 cm]\z1: $1021$},xshift=\xs,yshift=\ys,fill=blue,draw=black,scale=\x1](A29) at (-0.75,2.75) {};
\node [label={[label distance=\y1 cm]\z1: $1031$},xshift=\xs,yshift=\ys,circle,fill=red,draw=black,scale=\x1](A30) at (-1.5,3.5) {};
\node [label={[label distance=\y1 cm]\z1: $1020$},xshift=\xs,yshift=\ys,circle,fill=red,draw=black,scale=\x1](A31) at (0.25,3.75) {};
\node [label={[label distance=\y1 cm]\z1: $1030$},xshift=\xs,yshift=\ys,fill=blue,draw=black,scale=\x1](A32) at (-0.5,4.5) {};

\node [label={[label distance=\y1 cm]\z1: $0201$},xshift=-0.5*\xs,yshift=-0.5*\ys,circle,fill=red,draw=black,scale=\x1](A33) at (1,0) {};
\node [label={[label distance=\y1 cm]\z1: $0200$},xshift=-0.5*\xs,yshift=-0.5*\ys,fill=blue,draw=black,scale=\x1](A34) at (2,1) {};
\node [label={[label distance=\y1 cm]\z1: $0211$},xshift=-0.5*\xs,yshift=-0.5*\ys,fill=blue,draw=black,scale=\x1](A35) at (1,1) {};
\node [label={[label distance=\y1 cm]\z1: $0210$},xshift=-0.5*\xs,yshift=-0.5*\ys,circle,fill=red,draw=black,scale=\x1](A36) at (2,2) {};
\node [label={[label distance=\y1 cm]\z1: $0202$},xshift=-0.5*\xs,yshift=-0.5*\ys,fill=blue,draw=black,scale=\x1](A37) at (0.25,-0.75) {};
\node [label={[label distance=\y1 cm]\z1: $0203$},xshift=-0.5*\xs,yshift=-0.5*\ys,circle,fill=red,draw=black,scale=\x1](A38) at (-0.5,-1.5) {};

\node [label={[label distance=\y1 cm]\z1: $0301$},xshift=-\xs,yshift=-\ys,fill=blue,draw=black,scale=\x1](A39) at (1,0) {};
\node [label={[label distance=\y1 cm]\z1: $0300$},xshift=-\xs,yshift=-\ys,circle,fill=red,draw=black,scale=\x1](A40) at (2,1) {};
\node [label={[label distance=\y1 cm]\z1: $0311$},xshift=-\xs,yshift=-\ys,circle,fill=red,draw=black,scale=\x1](A41) at (1,1) {};
\node [label={[label distance=\y1 cm]\z1: $0310$},xshift=-\xs,yshift=-\ys,fill=blue,draw=black,scale=\x1](A42) at (2,2) {};
\node [label={[label distance=\y1 cm]\z1: $0302$},xshift=-\xs,yshift=-\ys,circle,fill=red,draw=black,scale=\x1](A43) at (0.25,-0.75) {};
\node [label={[label distance=\y1 cm]\z1: $0303$},xshift=-\xs,yshift=-\ys,fill=blue,draw=black,scale=\x1](A44) at (-0.5,-1.5) {};

xshift=-\xs,yshift=-\ys,

\draw [line width=\w1 mm] (A1)--(A2)--(A3)--(A4)--(A1)--(A5)--(A6)--(A7)--(A8)--(A5)--(A15)--(A16)--(A14)--(A13)--(A6) (A2)--(A9)--(A10)--(A12)--(A11)--(A3) (A2)--(A6) (A3)--(A7)  (A4)--(A8)  (A9)--(A11) (A13)--(A15)  (A17)--(A18)--(A19)--(A20)--(A17)--(A21)--(A22)--(A23)--(A24)--(A21)--(A31)--(A32)--(A30)--(A29)--(A22) (A18)--(A25)--(A26)--(A28)--(A27)--(A19) (A18)--(A22) (A19)--(A23)  (A20)--(A24)  (A25)--(A27) (A29)--(A31)  (A33)--(A34)--(A36)--(A35)--(A33)--(A37)--(A38) (A39)--(A40)--(A42)--(A41)--(A39)--(A43)--(A44) ; 

\draw [line width=\w1 mm,dashed, opacity=0.2] (A1)--(A17) (A2)--(A18) (A19)--(A3)--(A33)--(A39) (A20)--(A4)--(A34)--(A40) (A5)--(A21)  (A6)--(A22) (A23)--(A7)--(A35)--(A41) (A24)--(A8)--(A36)--(A42)  (A25)--(A9) (A26)--(A10) (A27)--(A11)--(A37)--(A43) (A28)--(A12)--(A38)--(A44) (A13)--(A29) (A14)--(A30) (A15)--(A31) (A16)--(A32);
\end{tikzpicture}  
\caption{The canonical decomposition and a proper $2$-coloring of $\Pi_4^{2,2}=\Pi_3^{2,2}\oplus\Pi_3^{2,2}\oplus\Pi_2^{2,2}\oplus\Pi_2^{2,2}$.} \label{fig:canon dec and coloring of Pi_4^{2,2}}
\end{figure}
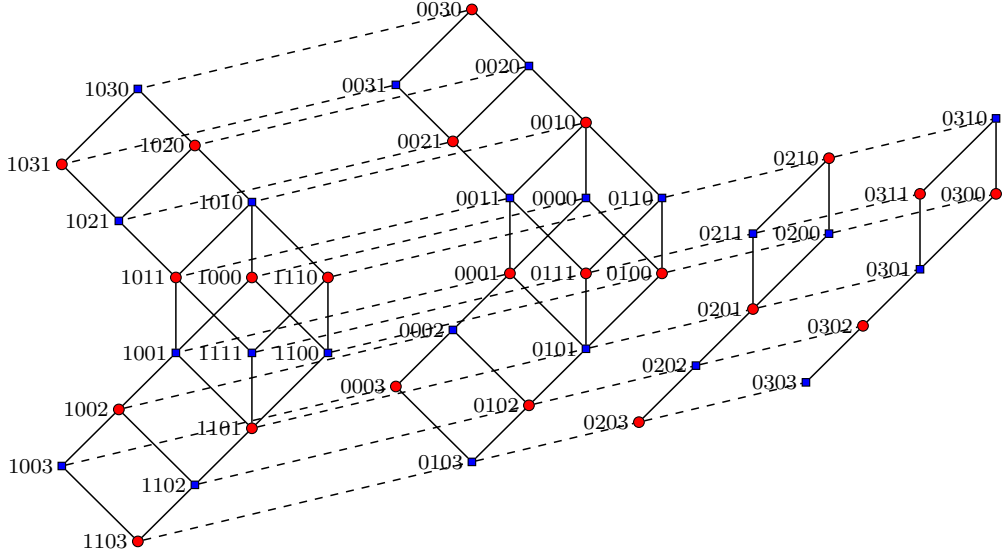

We now present another decomposition of the Horadam cubes.
To this end, recall that $P_a\square P_b$ denotes the Cartesian product of paths $P_a$ and $P_b$, and $P_a^k$ denotes the Cartesian product of path $P_a$ with itself $k$ times, that is, $P_{a}\square P_{a}\cdots \square P_{a}$. Also, note that
$|V(P_a^k\square P^m_b)|=a^kb^m$. The Cartesian products of the path graphs are called {\em grids} or {\em lattices}.
The following theorem illustrates how the combinatorial meaning of the Fibonacci 
polynomials of the identity (\ref{eq:horadam_number of vertices}) extends to the Horadam cubes. Before we state the next theorem, we need to define a map $\rho$. 
Let $\mathcal{F}_n$ be the set that contains all Fibonacci strings of length $n$ and recall that $|\mathcal{F}_n|=F_{n+2}$. We define $\rho:V(\Pi^{a,b}_n)\to\mathcal{F}_n$ to be the map on the alphabet
$\left\lbrace0,1,\dots,a+b-1\right\rbrace$ and extended to $V(\Pi^{a,b}_n)$ by
concatenation, as follows
$$\rho(\alpha)=\begin{cases} 0, & 0\leq \alpha \leq a-1,\\
		1, & \alpha\geq a.\end{cases} $$ 
For a Fibonacci string $w$, let $\rho^{-1}(w)$ denote the subgraph of
$\Pi^{a,b}_n$ induced by vertices
$\left\lbrace v\in V(\Pi^{a,b}_n)| \rho(v)=w\right\rbrace$.

\begin{theorem}
The Horadam cube $\Pi^{a,b}_n$ can be decomposed into $F_{n+1}$ grids, where $F_n$ denotes the $n$-th Fibonacci number. 
If $\Pi^{a,b}_n/\rho$ denotes the quotient graph of $\Pi^{a,b}_n$
obtained by identifying all vertices which are identified by $\rho$, and
two vertices $w_1$ and $w_2$ in $\Pi^{a,b}_n/\rho$ are adjacent if there is at
least one edge in $\Pi^{a,b}_n$ connecting blocks $\rho^{-1}(w_1)$ and
$\rho^{-1}(w_2)$. Then $\Pi^{a,b}_n/\rho$ is isomorphic to the Fibonacci
cube $\Gamma_{n-1}$. \label{tm:decomposition II} \end{theorem}

\begin{proof} Let $w$ be any Fibonacci string  that contains $k$ ones. Since every letter smaller than or equal to $a-1$ maps into $0$, and every block $0l$ for $a\leq l\leq a+b-1$ maps into the block $01$, the set $\rho^{-1}(w)$ contains all vertices that contain some block $0l$ for $a\leq l\leq a+b-1$ in the same places where $w$ has the block $01$. The remaining positions can be filled with any letter between $0$ and $a-1$. Thus, the set $\rho^{-1}(w)$ induces a subgraph of $\Pi^{a,b}_n$ isomorphic to $P_a^{n-2k}P^k_b$. Also note that $\rho$ maps $V(\Pi^{a,b}_n)$ into Fibonacci strings that start with $0$, and that any Fibonacci string of length $n$ that starts with $0$ can be realized as the image of some string from $V(\Pi^{a,b}_n)$. More precisely, for an arbitrary Fibonacci string $w$, we can take, for example, the string $v$ that contains block $0a$ in all the same positions where the string $w$ contains the block $01$, and the remaining positions are filled with $0$. Thus, the image of the map $\rho$ contains all Fibonacci strings of length $n$ that start with $0$, and there are $F_{n+1}$ such strings. Hence, $|\rho(V(\Pi^{a,b}_n))|=F_{n+1}$. Furthermore, for every two different $w,v\in\mathcal{F}_n$, subgraphs $\rho^{-1}(w)$ and $\rho^{-1}(v)$ are vertex disjoint. That gives a decomposition into $F_{n+1}$ grids, namely for each Fibonacci string $w$, we have the grid $\rho^{-1}(w)$. If two Fibonacci strings $w_1$ and $w_2$ differ in only one position then there is at least one edge connecting the grids $\rho^{-1}(w_1)$ and $\rho^{-1}(w_2)$. Indeed, without loss of generality, let $w_1$ be the Fibonacci string with $1$ on some position $l$, and $w_2$ be the same string except for $0$ on position $l$. Now, let $v_1$ be the vertex in $\Pi^{a,b}_n$ that has the letter $a$ at the same positions where $w_1$ has ones, and $v_2$ be the same string as $v_1$ except for $a-1$ on position $l$. Then $\rho(v_1)=w_1$ and $\rho(v_2)=w_2$, and the edge between the grids $\rho^{-1}(w_1)$ and $\rho^{-1}(w_2)$ is the edge connecting $v_1$ and $v_2$. On the other hand, if two grids $\rho^{-1}(w_1)$ and $\rho^{-1}(w_2)$ have an edge connecting them, it is clear that $w_1$ and $w_2$ must differ in only one position. The claim follows.
\end{proof}

Figure \ref{fig:dec of Pi_3^(3,2) and Jacobsthal} shows the decomposition of $\Pi_3^{3,2}$ and the Jacobsthal cube $\Pi_5^{1,2}$. Note that $\Pi_3^{3,2}/\rho=\Gamma_{2}$ and $\Pi_5^{1,2}/\rho=\Gamma_4$.

\begin{figure}[h!] \centering \begin{tikzpicture}[scale=0.6]
\tikzmath{\x1 = 0.35; \y1 =-0.05; \z1=180; \w1=0.2; \xs=-8; \ys=0; \yss=-5;
\x2 = \x1 + 1; \y2 =\y1 +3; } 
\small
\node [label={[label distance=\y1 cm]\z1: $000$},circle,fill=blue,draw=black,scale=\x1](A1) at (0,4) {};
\node [label={[label distance=\y1 cm]\z1: $010$},circle,fill=blue,draw=black,scale=\x1](A2) at (0,5) {};
\node [label={[label distance=\y1 cm]\z1: $020$},circle,fill=blue,draw=black,scale=\x1](A3) at (0,6) {};
\node [label={[label distance=\y1 cm]\z1: $001$},circle,fill=blue,draw=black,scale=\x1](A4) at (-1.5,2.5) {};
\node [label={[label distance=\y1 cm]\z1: $011$},circle,fill=blue,draw=black,scale=\x1](A5) at (-1.5,3.5) {};
\node [label={[label distance=\y1 cm]\z1: $021$},circle,fill=blue,draw=black,scale=\x1](A6) at (-1.5,4.5) {};
\node [label={[label distance=\y1 cm]\z1: $002$},circle,fill=blue,draw=black,scale=\x1](A7) at (-3,1) {};
\node [label={[label distance=\y1 cm]\z1: $012$},circle,fill=blue,draw=black,scale=\x1](A8) at (-3,2) {};
\node [label={[label distance=\y1 cm]\z1: $022$},circle,fill=blue,draw=black,scale=\x1](A9) at (-3,3) {};
\node [label={[label distance=\y1 cm]\z1: $100$},circle,fill=blue,draw=black,scale=\x1](A10) at (1.5,2.5) {};
\node [label={[label distance=\y1 cm]\z1: $110$},circle,fill=blue,draw=black,scale=\x1](A11) at (1.5,3.5) {};
\node [label={[label distance=\y1 cm]\z1: $120$},circle,fill=blue,draw=black,scale=\x1](A12) at (1.5,4.5) {};
\node [label={[label distance=\y1 cm]\z1: $101$},circle,fill=blue,draw=black,scale=\x1](A13) at (0,1) {};
\node [label={[label distance=\y1 cm]\z1: $111$},circle,fill=blue,draw=black,scale=\x1](A14) at (0,2) {};
\node [label={[label distance=\y1 cm]\z1: $121$},circle,fill=blue,draw=black,scale=\x1](A15) at (0,3) {};
\node [label={[label distance=\y1 cm]\z1: $102$},circle,fill=blue,draw=black,scale=\x1](A16) at (-1.5,-0.5) {};
\node [label={[label distance=\y1 cm]\z1: $112$},circle,fill=blue,draw=black,scale=\x1](A17) at (-1.5,0.5) {};
\node [label={[label distance=\y1 cm]\z1: $122$},circle,fill=blue,draw=black,scale=\x1](A18) at (-1.5,1.5) {};
\node [label={[label distance=\y1 cm]\z1: $200$},circle,fill=blue,draw=black,scale=\x1](A19) at (3,1) {};
\node [label={[label distance=\y1 cm]\z1: $210$},circle,fill=blue,draw=black,scale=\x1](A20) at (3,2) {};
\node [label={[label distance=\y1 cm]\z1: $220$},circle,fill=blue,draw=black,scale=\x1](A21) at (3,3) {};
\node [label={[label distance=\y1 cm]\z1: $201$},circle,fill=blue,draw=black,scale=\x1](A22) at (1.5,-0.5) {};
\node [label={[label distance=\y1 cm]\z1: $211$},circle,fill=blue,draw=black,scale=\x1](A23) at (1.5,0.5) {};
\node [label={[label distance=\y1 cm]\z1: $221$},circle,fill=blue,draw=black,scale=\x1](A24) at (1.5,1.5) {};
\node [label={[label distance=\y1 cm]\z1: $202$},circle,fill=blue,draw=black,scale=\x1](A25) at (0,-2) {};
\node [label={[label distance=\y1 cm]\z1: $212$},circle,fill=blue,draw=black,scale=\x1](A26) at (0,-1) {};
\node [label={[label distance=\y1 cm]\z1: $222$},circle,fill=blue,draw=black,scale=\x1](A27) at (0,0) {};
\node [label={[label distance=\y1 cm]\z1: $003$},circle,fill=blue,draw=black,scale=\x1](A28) at (-4.5,-0.5) {};
\node [label={[label distance=\y1 cm]\z1: $103$},circle,fill=blue,draw=black,scale=\x1](A29) at (-3,-2) {};
\node [label={[label distance=\y1 cm]\z1: $203$},circle,fill=blue,draw=black,scale=\x1](A30) at (-1.5,-3.5) {};
\node [label={[label distance=\y1 cm]\z1: $030$},circle,fill=blue,draw=black,scale=\x1](A31) at (-1.5,7.5) {};
\node [label={[label distance=\y1 cm]\z1: $031$},circle,fill=blue,draw=black,scale=\x1](A32) at (-3,6) {};
\node [label={[label distance=\y1 cm]\z1: $032$},circle,fill=blue,draw=black,scale=\x1](A33) at (-4.5,4.5) {};
\node [label={[label distance=\y1 cm]\z1: $040$},circle,fill=blue,draw=black,scale=\x1](A34) at (-3,9) {};
\node [label={[label distance=\y1 cm]\z1: $041$},circle,fill=blue,draw=black,scale=\x1](A35) at (-4.5,7.5) {};
\node [label={[label distance=\y1 cm]\z1: $042$},circle,fill=blue,draw=black,scale=\x1](A36) at (-6,6) {};
\node [label={[label distance=\y1 cm]\z1: $004$},circle,fill=blue,draw=black,scale=\x1](A37) at (-6,-2) {};
\node [label={[label distance=\y1 cm]\z1: $104$},circle,fill=blue,draw=black,scale=\x1](A38) at (-4.5,-3.5) {};
\node [label={[label distance=\y1 cm]\z1: $204$},circle,fill=blue,draw=black,scale=\x1](A39) at (-3,-5) {};

\draw [line width=\w1 mm] (A1)--(A2)--(A3) (A4)--(A5)--(A6) (A7)--(A8)--(A9) (A10)--(A11)--(A12) (A13)--(A14)--(A15) (A16)--(A17)--(A18) (A19)--(A20)--(A21) (A22)--(A23)--(A24) (A25)--(A26)--(A27) (A1)--(A4)--(A7) (A2)--(A5)--(A8) (A3)--(A6)--(A9) (A10)--(A13)--(A16) (A11)--(A14)--(A17) (A12)--(A15)--(A18) (A19)--(A22)--(A25) (A20)--(A23)--(A26) (A21)--(A24)--(A27) (A31)--(A32)--(A33) (A34)--(A35)--(A36) (A28)--(A37) (A29)--(A38) (A30)--(A39) (A1)--(A10)--(A19) (A2)--(A11)--(A20) (A3)--(A12)--(A21) (A4)--(A13)--(A22) (A5)--(A14)--(A23) (A6)--(A15)--(A24) (A7)--(A16)--(A25) (A8)--(A17)--(A26) (A9)--(A18)--(A27)  (A28)--(A29)--(A30) (A37)--(A38)--(A39) (A31)--(A34) (A32)--(A35) (A33)--(A36); 

\draw [line width=\w1 mm,dashed, opacity=0.2] (A31)--(A3) (A32)--(A6) (A33)--(A9) (A7)--(A28) (A16)--(A29) (A25)--(A30); 

\end{tikzpicture}\centering \raisebox{0.3\height}{\begin{tikzpicture}[scale=0.6]
\tikzmath{\x1 = 0.35; \y1 =-0.05; \z1=90; \w1=0.2;   \xs=35; \ys=0;
\x2 = \x1 + 1; \y2 =\y1 +3; } 
\small

\node [label={[label distance=\y1 cm]\z1: $00000$},circle,fill=blue,draw=black,scale=\x1](A3) at (0,6) {};
\node [label={[label distance=\y1 cm]\z1: $00001$},circle,fill=blue,draw=black,scale=\x1](A2) at (1.5,7.5) {};
\node [label={[label distance=\y1 cm]\z1: $00002$},circle,fill=blue,draw=black,scale=\x1](A1) at (3,9) {};
\node [label={[label distance=\y1 cm]\z1: $00010$},circle,fill=blue,draw=black,scale=\x1](A4) at (1.5,4.5) {};
\node [label={[label distance=\y1 cm]\z1: $00020$},circle,fill=blue,draw=black,scale=\x1](A5) at (3,3) {};
\node [label={[label distance=\y1 cm]\z1: $00100$},circle,fill=blue,draw=black,scale=\x1](A6) at (-1.5,7.5) {};
\node [label={[label distance=\y1 cm]\z1: $00101$},circle,fill=blue,draw=black,scale=\x1](A7) at (0,9) {};
\node [label={[label distance=\y1 cm]\z1: $00102$},circle,fill=blue,draw=black,scale=\x1](A8) at (1.5,10.5) {};
\node [label={[label distance=\y1 cm]\z1: $00200$},circle,fill=blue,draw=black,scale=\x1](A9) at (-3,9) {};
\node [label={[label distance=\y1 cm]\z1: $00201$},circle,fill=blue,draw=black,scale=\x1](A10) at (-1.5,10.5) {};
\node [label={[label distance=\y1 cm]\z1: $00202$},circle,fill=blue,draw=black,scale=\x1](A11) at (0,12) {};

\node [label={[label distance=\y1 cm]\z1: $01000$},xshift=\xs,yshift=\ys,circle,fill=blue,draw=black,scale=\x1](A14) at (0,6) {};
\node [label={[label distance=\y1 cm]\z1: $01001$},xshift=\xs,yshift=\ys,circle,fill=blue,draw=black,scale=\x1](A13) at (1.5,7.5) {};
\node [label={[label distance=\y1 cm]\z1: $01002$},xshift=\xs,yshift=\ys,circle,fill=blue,draw=black,scale=\x1](A12) at (3,9) {};
\node [label={[label distance=\y1 cm]\z1: $01010$},xshift=\xs,yshift=\ys,circle,fill=blue,draw=black,scale=\x1](A15) at (1.5,4.5) {};
\node [label={[label distance=\y1 cm]\z1: $01020$},xshift=\xs,yshift=\ys,circle,fill=blue,draw=black,scale=\x1](A16) at (3,3) {};

\node [label={[label distance=\y1 cm]\z1: $02000$},xshift=2*\xs,yshift=2*\ys,circle,fill=blue,draw=black,scale=\x1](A19) at (0,6) {};
\node [label={[label distance=\y1 cm]\z1: $02001$},xshift=2*\xs,yshift=2*\ys,circle,fill=blue,draw=black,scale=\x1](A18) at (1.5,7.5) {};
\node [label={[label distance=\y1 cm]\z1: $02002$},xshift=2*\xs,yshift=2*\ys,circle,fill=blue,draw=black,scale=\x1](A17) at (3,9) {};
\node [label={[label distance=\y1 cm]\z1: $02010$},xshift=2*\xs,yshift=2*\ys,circle,fill=blue,draw=black,scale=\x1](A20) at (1.5,4.5) {};
\node [label={[label distance=\y1 cm]\z1: $02020$},xshift=2*\xs,yshift=2*\ys,circle,fill=blue,draw=black,scale=\x1](A21) at (3,3) {};
\draw [line width=\w1 mm] (A1)--(A2)  (A4)--(A5) (A7)--(A8) (A10)--(A11)  (A6)--(A9) (A7)--(A10) (A8)--(A11) (A12)--(A13)  (A15)--(A16) (A17)--(A18)  (A20)--(A21)  (A12)--(A17) (A13)--(A18) (A14)--(A19) (A15)--(A20) (A16)--(A21); 
\draw [line width=\w1 mm,dashed, opacity=0.2] (A9)--(A10) (A6)--(A7) (A2)--(A7) (A1)--(A8) (A1)--(A12) (A2)--(A13) (A2)--(A3)  (A3)--(A6) (A3)--(A4) (A4)--(A15) (A5)--(A16) (A3)--(A14) (A13)--(A14) (A18)--(A19) (A14)--(A15) (A19)--(A20); 
\end{tikzpicture}}
\caption{The decomposition of $\Pi_3^{3,2}=P^{3}_3\oplus 2P_2\square P_3$ and Jacobsthal cube $\Pi_5^{1,2}=P^{5}_1\oplus 4 P^3_1\square P_2\oplus 3P^2_2$.} \label{fig:dec of Pi_3^(3,2) and Jacobsthal}
\end{figure}
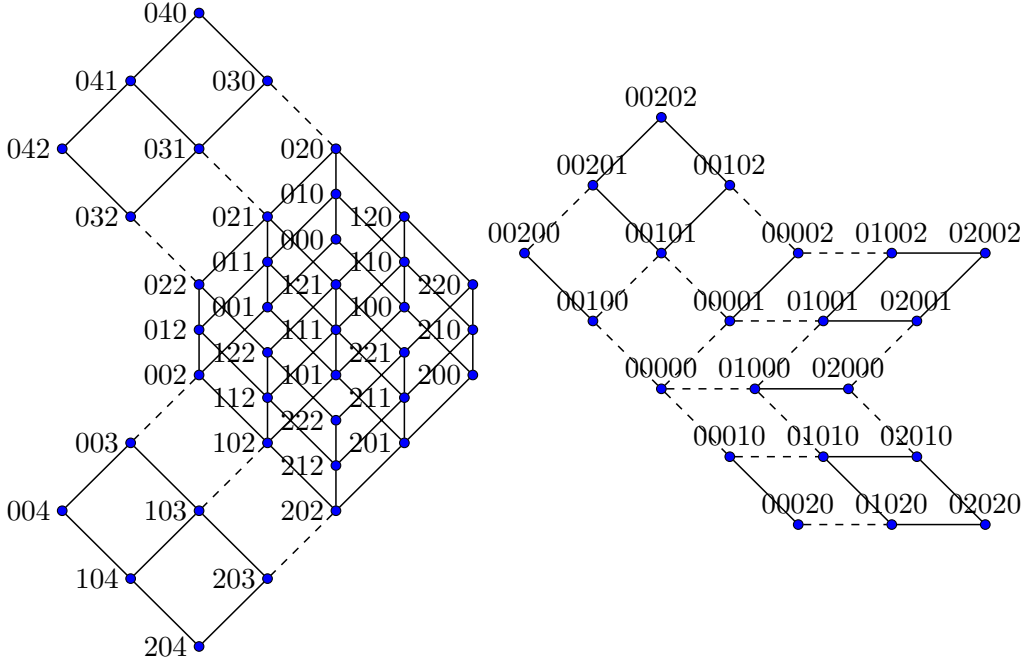


\section{Number of edges}

The canonical decomposition from Theorem \ref{tm:candec_hor} is a useful tool to obtain the recursive relation for the number of edges. Let $e^{a,b}_n$ denote the number of edges in the Horadam cube $\Pi^{a,b}_n$.
Since $\Pi^{a,b}_1=P_a$ and $\Pi^{a,b}_2$ is $a\times a$ grid with the path $P_b$ appended to the vertex $0(a-1)$, we have $e_1^{a,b}=a-1$ and $e_2^{a,b}=2a^2-2a+b$. Using the same approach as with the metallic cubes, one can easily see that the graph $\Pi^{a,b}_n$ contains $a$ copies of the graph $\Pi^{a,b}_{n-1}$ and $b$ copies of the graph $\Pi^{a,b}_{n-2}$. Thus, the contribution of those subgraphs to the number of edges is $ae_{n-1}^{a,b}+be_{n-2}^{a,b}$. Furthermore, there are $(a-1)s^{a,b}_{n-1}$ edges connecting $a$ copies of the graph $\Pi^{a,b}_{n-1}$ and $bs^{a,b}_{n-2}$ edges connecting $b$ copies of the graph $\Pi^{a,b}_{n-2}$. This simple analysis shows that the recursive relation for the number of edges of the metallic cubes naturally expands to the number of edges of the Horadam cubes. We have
\begin{align}\label{eq:Horadam_number_of edges_recursive}
e^{a,b}=ae_{n-1}^{a,b}+be_{n-2}^{a,b}+s^{a,b}_n-s^{a,b}_{n-1}.
\end{align}  

Recurrence (\ref{eq:Horadam_number_of edges_recursive}) immediately yields the generating function $E(x)=\sum_{n\geq 0}e^{a,b}_nx^n$ for the number of edges in Horadam cubes: \begin{align}\label{eq:hor_number_of_edg_gen}
E(x)=\dfrac{(a-1)x+bx^2}{(1-ax-bx^2)^2}.
\end{align}
Using recurrence (\ref{eq:Horadam_number_of edges_recursive}), we list the first values of the sequence $e^{a,b}_n$ in Table \ref{table:horadam_number of edges}.
\begin{table}[h]\centering$\begin{array}{r|l}
n & e^a_n\\
\hline
1 & a-1\\
2 & 2a^2-2a+b\\
3 & 3a^3-3a^2+4ab-2b\\
4 & 4a^4-4a^3+9a^2b-6ab+2b^2\\
5 & 5a^5-5a^4+16a^3b-12a^2b+9ab^2-3b^2\\
6 & 6a^6-6a^5+25a^4b-20a^3b+24a^2b^2-12ab^2+3b^3
\end{array}$\caption{The number of edges in the Horadam cubes.}\label{table:horadam_number of edges}
\end{table}

Now we can state the theorem that expresses the numbers $e^{a,b}_n$ in terms of the number of vertices in the Horadam cubes.

\begin{theorem}\label{tm:horadam_number_of_edges_using_s_n}
The number of edges in the Horadam cube $\Pi^{a,b}_n$ is
\begin{equation*}
e^{a,b}_n=\sum_{k=0}^{n-1} s^{a,b}_k \left(s^{a,b}_{n-k}-s^{a,b}_{n-1-k}\right).
\end{equation*} 
\end{theorem}
\begin{proof}
Using generating functions (\ref{eq:Horadam_seq_generating_function}) and (\ref{eq:hor_number_of_edg_gen}) and the fact that $$ S^2(x)=\sum_{n\geq 0}\sum_{k=0}^n s^{a,b}_k s^{a,b}_{n-k}x^n$$ we have \begin{align*}
E(x)&=\dfrac{(a-1)x+bx^2}{(1-ax-bx^2)^2}\\
&=((a-1)x+bx^2)S^2(x)\\
&=a\sum_{n\geq 0}\sum_{k=0}^n s^{a,b}_k s^{a,b}_{n-k}x^{n+1}+b\sum_{n\geq 0}\sum_{k=0}^n s^{a,b}_k s^{a,b}_{n-k}x^{n+2}-\sum_{n\geq 0}\sum_{k=0}^n s^{a,b}_k s^{a,b}_{n-k}x^{n+1}\\
&=a\sum_{n\geq 1}\sum_{k=0}^{n-1} s^{a,b}_k s^{a,b}_{n-1-k}x^{n}+b\sum_{n\geq 2}\sum_{k=0}^{n-2} s^{a,b}_k s^{a,b}_{n-2-k}x^{n}-\sum_{n\geq 1}\sum_{k=0}^{n-1} s^{a,b}_k s^{a,b}_{n-1-k}x^{n}\\
&=\sum_{n\geq 0}\sum_{k=0}^{n-1} s^{a,b}_k \left(as^{a,b}_{n-1-k}+b s^{a,b}_{n-2-k}-s^{a,b}_{n-1-k}\right)x^n\\
&=\sum_{n\geq 0}\sum_{k=0}^{n-1} s^{a,b}_k \left(s^{a,b}_{n-k}-s^{a,b}_{n-1-k}\right)x^n. 
\end{align*}    \end{proof}

For $a=b=1$ we have $s^{1,1}_{n}=F_{n+1}$ and \begin{align*} e^{1,1}_n&=\sum_{k=0}^{n-1} s^{1,1}_k \left(s^{1,1}_{n-k}-s^{1,1}_{n-1-k}\right)\\ &=\sum_{k=0}^{n-1} s^{1,1}_k s^{1,1}_{n-k-2}\\
&=\sum_{k=0}^{n} F_{k} F_{n-k}, \end{align*} which immediately retrieves the result on the number of edges in the Fibonacci cubes \cite{Klavzar}. Theorem \ref{tm:horadam_number_of_edges_using_s_n} allows us to extend this result for $a=1$ and $b\geq 1$ and we obtain \begin{align*}
e^{1,b}_n=b\sum\limits_{k=0}^{n-1}s^{1,b}_{k-1}s^{1,b}_{n-k-1}.
\end{align*}  In particular, for the Jacobsthal cubes we obtained the sequence \seqnum{A095977}: $$e^{1,2}_n=2\sum_{k=0}^nJ_kJ_{n-k}.$$

Another way to express the number of edges is by using binomial coefficients. 

\begin{theorem}The number of edges in the Horadam cube $\Pi^{a,b}_n$ is
\begin{equation*}
e^{a,b}_n=\sum_{k=0}^n(-1)^{n-k}\ceil*{\frac{n+k}{2}}\binom{\floor*{\frac{n+k}{2}}}{k}a^kb^{\floor*{\frac{n-k}{2}}}.
\end{equation*} 
\end{theorem} 

\begin{proof}
It is easy to verify that the statement holds for $n=1$ and $n=2$. We proceed by induction.
Since the number of vertices $s^{a,b}_n$ satisfies the identity (\ref{eq:horadam_number of vertices}), we have
\begin{align*}
s^{a,b}_{n}-s^{a,b}_{n-1}&=\sum\limits_{k\geq 0}\binom{n-k}{k}a^{n-2k}b^{k}-\sum\limits_{k\geq 0}\binom{n-k-1}{k}a^{n-2k-1}b^{k}\\
&=\sum\limits_{k=0}^n(-1)^k\binom{n-\ceil*{\frac{k}{2}}}{\floor*{\frac{k}{2}}}a^{n-k}b^{\floor*{\frac{k}{2}}}\\
&=\sum\limits_{k= 0}^n(-1)^k\binom{n-\ceil*{\frac{k}{2}}}{n-k}a^{n-k}b^{\floor*{\frac{k}{2}}}\\
&=\sum\limits_{k=0}^n(-1)^{n-k}\binom{n-\ceil*{\frac{n-k}{2}}}{k}a^{k}b^{\floor*{\frac{n-k}{2}}}\\
&=\sum\limits_{k=0}^n(-1)^{n-k}\binom{\floor*{\frac{n+k}{2}}}{k}a^{k}b^{\floor*{\frac{n-k}{2}}}.
\end{align*}
By using the inductive hypothesis, after adjusting indices and expanding the
 range of summation, we obtain
\begin{align*}
a\cdot e^{a,b}_{n-1}+b\cdot e^{a,b}_{n-2}=&\sum_{k=0}^{n-1}(-1)^{n-k-1}\ceil*{\frac{n+k-1}{2}}\binom{\floor*{\frac{n+k-1}{2}}}{k}a^{k+1}b^{\floor*{\frac{n-k-1}{2}}}+\\&+\sum_{k=0}^{n-2}(-1)^{n-k}\ceil*{\frac{n+k-2}{2}}\binom{\floor*{\frac{n+k-2}{2}}}{k}a^{k}b^{\floor*{\frac{n-k}{2}}}\\
=&\sum_{k=1}^{n}(-1)^{n-k}\ceil*{\frac{n+k-2}{2}}\binom{\floor*{\frac{n+k-2}{2}}}{k-1}a^{k}b^{\floor*{\frac{n-k}{2}}}+\\&+\sum_{k=0}^{n}(-1)^{n-k}\ceil*{\frac{n+k-2}{2}}\binom{\floor*{\frac{n+k-2}{2}}}{k}a^{k}b^{\floor*{\frac{n-k}{2}}}\\
=&\sum_{k=0}^{n}(-1)^{n-k}\ceil*{\frac{n+k-2}{2}}\binom{\floor*{\frac{n+k}{2}}}{k}a^{k}b^{\floor*{\frac{n-k}{2}}}.
\end{align*}
Now, by using expressions for $s^a_{n}-s^a_{n-1}$ and
$a\cdot e^a_{n-1}+b\cdot e^a_{n-2}$ our claim follows at once. \end{proof}

Note that for $a=2$, by the equation (\ref{eq:Horadam_seq_generating_function}), we have  $S'(x)=\dfrac{2+2bx}{(1-2x-bx^2)^2}$ and 
$$E(x)=\dfrac{1}{2}xS'(x).$$ Hence, we can expand the result obtained by Munarini for all $b\geq 1$, to obtain $$e^{2,b}_n=\dfrac{n}{2}s^{2,b}_n.$$ Here is worth mentioning that, although the Pell graphs and Horadam cubes $\Pi^{2,1}_n$ are not isomorphic, they do share the same number of vertices and edges. In particular, if $\Pi_n$ denotes the Pell graph of dimension $n$, we have $$|E(\Pi^{2,1}_n)|=\sum_{k=0}^n(-1)^{n-k}\ceil*{\frac{n+k}{2}}\binom{\floor*{\frac{n+k}{2}}}{k}2^k=\dfrac{nP_{n+1}}{2}=|E(\Pi_n)|.$$

\section{Embedding into hypercubes and median graphs}
In this section, we justify the "cube" part of the name by showing that Horadam cubes are, indeed, induced subgraphs of hypercubes.

\begin{theorem}
For any $a,b\geq 1$, the Horadam cube $\Pi^{a,b}_n$ is an induced subgraph of the hypercube $Q_{(a+b-1)n}$. \label{tm: Horadam_embedding}
\end{theorem} \begin{proof} We define a map $\sigma:\Pi^{a,b}_n\to Q_{(a+b-1)n}$ on the primitive blocks, i.e., the blocks every string from the set $\mathcal{S}^{a,b}_{n}$  can be uniquely decomposed into. Those blocks are $0,1,\dots,a-1, 0a, 0(a+1),\ldots, 0(a+b-1)$. We define $\sigma$ as follows: \begin{align*}
\sigma(0)&=\underbrace{011\cdots 111}_{a}  \underbrace{000\cdots000}_{b-1}\\
\sigma(1)&=\underbrace{001\cdots 111}_{a}  \underbrace{000\cdots000}_{b-1}\\
\sigma(2)&=\underbrace{000\cdots 111}_{a}  \underbrace{000\cdots000}_{b-1}\\
 &\vdots\\
\sigma(a-2)&=\underbrace{000\cdots 001}_{a}  \underbrace{000\cdots000}_{b-1}\\
\sigma(a-1)&=\underbrace{000\cdots 000}_{a}  \underbrace{000\cdots000}_{b-1}\\
\sigma(0a)&=\underbrace{111\cdots 111}_{a}  \underbrace{000\cdots000}_{b-1}\underbrace{000\cdots 000}_{a}  \underbrace{000\cdots000}_{b-1}\\
\sigma(0(a+1))&=\underbrace{111\cdots 111}_{a}  \underbrace{000\cdots000}_{b-1}\underbrace{000\cdots 000}_{a}  \underbrace{100\cdots000}_{b-1}\\
\sigma(0(a+2))&=\underbrace{111\cdots 111}_{a}  \underbrace{000\cdots000}_{b-1}\underbrace{000\cdots 000}_{a}  \underbrace{110\cdots000}_{b-1}\\
 &\vdots\\
 \sigma(0(a+b-1))&=\underbrace{111\cdots 111}_{a}  \underbrace{000\cdots000}_{b-1}\underbrace{000\cdots 000}_{a}  \underbrace{111\cdots 111}_{b-1}\\
\end{align*}
The string $\sigma(k),0\leq k\leq a-1$, of length $a+b-1$ starts with $k+1$ zeros and $a-1-k$ ones. The remaining $b-1$ letters are zeroes. On the other hand, the string $\sigma(0(a+l)),0\leq l\leq b-1$ has length $2(a+b-1)$ and contains $a$ ones followed by $a+b-1$ zeroes, $l$ ones and $b-1-l$ zeroes. The map $\sigma$, defined on the primitive blocks, extends by concatenation to $\sigma:\Pi^{a,b}_n\to Q_{(a+b-1)n}$.

The map $\sigma$ assigns different binary strings to each primitive block, and the neighboring letters are assigned to the binary strings that differ in only one position. We conclude that the map $\rho$ is injective and preserves adjacency, so we obtained a subgraph of $Q_{(a+b-1)n}$ isomorphic to the graph $\Pi^{a,b}_{n}$. \end{proof}

\begin{remark} For $a=3$ and $b=2$, we have $\sigma(0)=0110$, $\sigma(1)=0010$, $\sigma(2)=0000$, $\sigma(03)=11100000$ and $\sigma(04)=1110001$. For the Jacobsthal cube, where $a=1$ and $b=2$,  we have $\sigma(0)=00$, $\sigma(01)=1000$ and $\sigma(02)=1001$.
\end{remark}

\begin{remark} The dimension of the hypercube from Theorem \ref{tm: Horadam_embedding} is not the smallest possible. For example, the number of vertices in Jacobsthal cubes grows slower than the number of vertices in hypercubes. More precisely, $|V(\Pi^{1,2}_n)|=\frac{2}{3}2^n+\frac{1}{3}(-1)^n<2^n=|V(Q_n)|$. Thus, we can define $\sigma(0)=0$, $\sigma(01)=10$, and $\sigma(02)=11$ to obtain an embedding of the Jacobsthal cube into a hypercube of the same dimension. But this embedding does not yield an induced and median-closed subgraph of the hypercube $Q_n$, so, in Theorem \ref{tm: Horadam_embedding}, we constructed an embedding into a larger hypercube to obtain an induced and median-closed subgraph. For example, Figure \ref{fig:embedd of Jac into hypercube} shows an embedding of the Jacobsthal cube $\Pi^{1,2}_4$ into the hypercube $Q_4$. But vertices $0001$ and $0020$, not adjacent in $\Pi^{1,2}_4$, are mapped into adjacent vertices in $Q_4$. The same situation occurs with the vertices $0200$ and $0010$. Thus, although we obtained an embedding, we did not obtain an induced subgraph of $Q_4$. It would be interesting to determine the smallest hypercube dimension containing a Horadam cube as a subgraph.

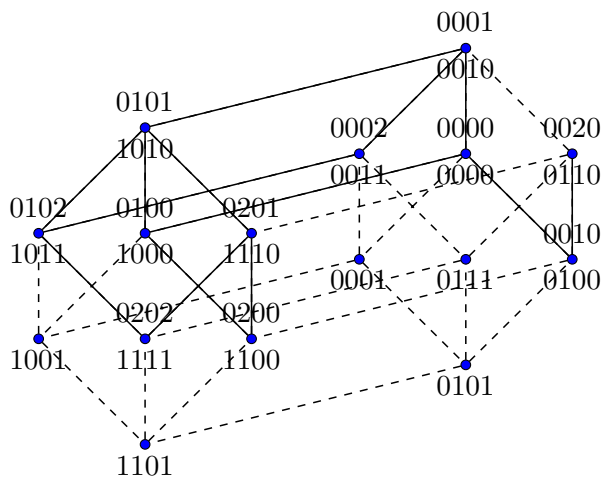
\begin{figure}[h!] \centering \begin{tikzpicture}[scale=1.4]
\tikzmath{\x1 = 0.35; \y1 =-0.05; \z1=-90; \z2=90; \w1=0.2; \xs=-120; \ys=-30; \yss=-5;
\x2 = \x1 + 1; \y2 =\y1 +3; } 
\small

\node [label={[label distance=\y1 cm]\z1: $0000$},circle,fill=blue,draw=black,scale=\x1](A1) at (1,2) {};
\node [label={[label distance=\y1 cm]\z1: $0001$},circle,fill=blue,draw=black,scale=\x1](A2) at (0,1) {};
\node [label={[label distance=\y1 cm]\z1: $0101$},circle,fill=blue,draw=black,scale=\x1](A3) at (1,0) {};
\node [label={[label distance=\y1 cm]\z1: $0100$},circle,fill=blue,draw=black,scale=\x1](A4) at (2,1) {};
\node [label={[label distance=\y1 cm]\z1: $0010$},circle,fill=blue,draw=black,scale=\x1](A5) at (1,3) {};
\node [label={[label distance=\y1 cm]\z1: $0011$},circle,fill=blue,draw=black,scale=\x1](A6) at (0,2) {};
\node [label={[label distance=\y1 cm]\z1: $0111$},circle,fill=blue,draw=black,scale=\x1](A7) at (1,1) {};
\node [label={[label distance=\y1 cm]\z1: $0110$},circle,fill=blue,draw=black,scale=\x1](A8) at (2,2) {};

\node [label={[label distance=\y1 cm]\z1: $1000$},xshift=\xs,yshift=\ys,circle,fill=blue,draw=black,scale=\x1](A9) at (1,2) {};
\node [label={[label distance=\y1 cm]\z1: $1001$},xshift=\xs,yshift=\ys,circle,fill=blue,draw=black,scale=\x1](A10) at (0,1) {};
\node [label={[label distance=\y1 cm]\z1: $1101$},xshift=\xs,yshift=\ys,circle,fill=blue,draw=black,scale=\x1](A11) at (1,0) {};
\node [label={[label distance=\y1 cm]\z1: $1100$},xshift=\xs,yshift=\ys,circle,fill=blue,draw=black,scale=\x1](A12) at (2,1) {};
\node [label={[label distance=\y1 cm]\z1: $1010$},xshift=\xs,yshift=\ys,circle,fill=blue,draw=black,scale=\x1](A13) at (1,3) {};
\node [label={[label distance=\y1 cm]\z1: $1011$},xshift=\xs,yshift=\ys,circle,fill=blue,draw=black,scale=\x1](A14) at (0,2) {};
\node [label={[label distance=\y1 cm]\z1: $1111$},xshift=\xs,yshift=\ys,circle,fill=blue,draw=black,scale=\x1](A15) at (1,1) {};
\node [label={[label distance=\y1 cm]\z1: $1110$},xshift=\xs,yshift=\ys,circle,fill=blue,draw=black,scale=\x1](A16) at (2,2) {};

\node [label={[label distance=\y1 cm]\z2: $0000$}](B1) at (1,2) {};
\node [label={[label distance=\y1 cm]\z2: $0010$}](B4) at (2,1) {};
\node [label={[label distance=\y1 cm]\z2: $0001$}](B5) at (1,3) {};
\node [label={[label distance=\y1 cm]\z2: $0002$}](B6) at (0,2) {};
\node [label={[label distance=\y1 cm]\z2: $0020$}](B8) at (2,2) {};

\node [label={[label distance=\y1 cm]\z2: $0100$},xshift=\xs,yshift=\ys](B9) at (1,2) {};
\node [label={[label distance=\y1 cm]\z2: $0200$},xshift=\xs,yshift=\ys](B12) at (2,1) {};
\node [label={[label distance=\y1 cm]\z2: $0101$},xshift=\xs,yshift=\ys](B13) at (1,3) {};
\node [label={[label distance=\y1 cm]\z2: $0102$},xshift=\xs,yshift=\ys](B14) at (0,2) {};
\node [label={[label distance=\y1 cm]\z2: $0202$},xshift=\xs,yshift=\ys](B15) at (1,1) {};
\node [label={[label distance=\y1 cm]\z2: $0201$},xshift=\xs,yshift=\ys](B16) at (2,2) {};

\draw [line width=\w1 mm,opacity=0.5,dashed] (A1)--(A2)--(A3)--(A4)--(A1)--(A5)--(A6)--(A7)--(A8)--(A5)
(A2)--(A6) (A3)--(A7)  (A4)--(A8)  (A10)--(A14) (A11)--(A15) (A12)--(A16) (A9)--(A10)--(A11)--(A12)--(A9)--(A13)--(A14)--(A15)--(A16)--(A13) (A1)--(A9) (A2)--(A10) (A3)--(A11)(A4)--(A12)(A5)--(A13)(A6)--(A14)(A7)--(A15) (A8)--(A16) ;

\draw [line width=\w1 mm] (A12)--(A16)--(A15) (A9)--(A13)--(A14) (A1)--(A5)--(A6)  (A15)--(A14)--(A6) (A16)--(A13)--(A5) (A12)--(A9)--(A1)--(A4)--(A8);

\end{tikzpicture}  
\caption{Embedding of the Jacobsthal cube $\Pi^{1,2}_4$ into the hypercube $Q_4$.} \label{fig:embedd of Jac into hypercube}
\end{figure}

\end{remark}

Recall that a {\em median} of three vertices is a vertex that lies on the shortest path
between every two of three vertices. We say that a graph $G$ is a {\em median
graph} if every three vertices of $G$ have a unique median. Since the Fibonacci cubes, Pell graphs, and metallic cubes are median graphs \cite{metallic, Klavzar, Munarini}, the next natural step is to check whether the Horadam cubes are also median graphs. We found the answer to be positive, and we present that in the next theorem. To do that, we use the following theorem \cite{Mulder}: \begin{theorem}[Mulder]
A graph $G$ is a median graph if and only if $G$ is a connected induced subgraph
of an $n$-cube such that with any three vertices of $G$ their median in
$n$-cube is also a vertex of $G$. 
\end{theorem}

The median of the triple in the hypercube $Q_n$ is obtained by the majority rule. If $\alpha=\epsilon_1\cdots\epsilon_n$,
$\beta=\delta_1\cdots\delta_n$ and $\gamma=\rho_1\cdots\rho_n$ are binary
strings, i.e., $\epsilon_i,\delta_i,\rho_i\in\left\lbrace 0,1\right\rbrace$,
then their median is $m=\zeta_i\cdots\zeta_n$, where $\zeta_i$ is equal to
the number that appears at least twice among the numbers $\epsilon_i$,
$\delta_i$ and $\rho_i$.

\begin{theorem} For any $a,b\geq 1$ and $n\geq 1$, the Horadam cube $\Pi^{a,b}_n$ is a median graph. 
\end{theorem}
\begin{proof}

Let $\sigma$ be the map defined as in the proof of
Theorem \ref{tm: Horadam_embedding}. The graph $\Pi^{a,b}_n$ is a connected
induced subgraph of a hypercube. To finish our proof, we just need
to verify that the subgraph induced by the set
$\sigma\left(\mathcal{S}^{a,b}_n\right)$  is median closed. Every string from 
the set $\sigma\left(\mathcal{S}^{a,b}_n\right)$ can be uniquely decomposed into
blocks $\sigma\left(k\right),0\leq k\leq a-1$ of length $a+b-1$, 
and  $\sigma\left(0(a+l)\right),0\leq l\leq b-1$ of length $2(a+b-1)$. First we consider three blocks, $\sigma\left(i\right)$,
$\sigma\left(j\right)$ and $\sigma\left(k\right)$. Without loss of
generality, we can assume that $0\leq i\leq j\leq k\leq a-1$. Then their median is
$\sigma\left(j\right)$. In the second case, we have $\sigma\left(i\right)$
and $\sigma\left(j\right)$ for $0\leq i\leq j\leq a-1$, and the first or the second half of the block $\sigma\left(0(a+l)\right)$. Their median is now
$\sigma\left(i\right)$. The remaining cases include all combinations of blocks; since they are done similarly, we omit the details. 
\end{proof}

 For example, when $a=3$ and $b=2$, we have $\sigma(0)=0110$, $\sigma(1)=0010$, $\sigma(2)=0000$, $\sigma(03)=11100000$, and $\sigma(04)=11100001$. Then the median of vertices $042$, $204$ and $110$ is $112$.

\section{Distribution of degrees}

Klav\v{z}ar et al. determined the distribution of degrees in Fibonacci and Lucas cubes \cite{degree}. Here we generalize the result for the Fibonacci cubes to all Horadam cubes. More precisely, in this section, we investigate the distribution of degrees by determining the recurrence relations and by computing the bivariate generating functions for two-indexed sequences $\Delta_{n,k}$ which count the number of vertices in
$\Pi_n^{a,b}$ with exactly $k$ neighbors. We suppress $a$ and $b$ to simplify
the notation. The distribution of degrees in Horadam cubes is considered in two separate cases, for $a=1$ and $a\geq 2$. The following two theorems provide recursive relations for $\Delta_{n,k}$.  

\begin{theorem}\label{tm: Horadam_distribution of degrees a>=2}
Let $a\geq 2$ and let $\Delta_{n,k}$ denote the number of vertices in the Horadam cube $\Pi^{a,b}_n$ with exactly $k$ neighbors. Then the sequence $\Delta_{n,k}$ satisfies the recurrence
\begin{align*}
\Delta_{n,k}=2\Delta_{n-1,k-1}+(a-2)\Delta_{n-1,k-2}+\Delta_{n-2,k-1}+(b-2)\Delta_{n-2,k-2}+\Delta_{n-2,k-3},
\end{align*} with initial values $\Delta_{0,0}=1$, $\Delta_{1,1}=2$ and $\Delta_{1,2}=a-2$. 
\end{theorem} \begin{proof}
Since the graph $\Pi^{a,b}_0$ is an empty graph, and $\Pi^{a,b}_1$ is a path graph $P_a$, one can easily verify the initial values. Let $n\geq 2$. By Theorem \ref{tm:candec_hor}, $\Pi^{a,b}_n$ contains $a$ copies of the graph $\Pi^{a,b}_{n-1}$. Recall that the copy of the subgraph $\Pi^{a,b}_{n-1}$ induced by vertices in $\Pi^{a,b}_n$ starting with the letter $k, 0\leq k\leq a-1$, is denoted by $k\Pi^{a,b}_{n-1}$ and the copy of the subgraph $\Pi^{a,b}_{n-2}$ induced by vertices in $\Pi^{a,b}_n$ starting with the block $0l, a\leq l\leq a+b-1$,  is denoted by $0l\Pi^{a,b}_{n-2}$. Each vertex in subgraphs $k\Pi^{a,b}_{n-1}$, where $1\leq k\leq a-2$, has exactly two new neighbors, namely the corresponding vertices in adjacent subgraphs, $(k-1)\Pi^{a,b}_{n-1}$ and $(k+1)\Pi^{a,b}_{n-1}$. Thus the degree of each vertex in those $a-2$ subgraphs is increased by two and they contribute to $\Delta_{n,k}$ with $(a-2)\Delta_{n-1,k-2}$. Furthermore, the vertices of the copy of $\Pi^{a,b}_{n-1}$ induced by vertices starting with $a-1$, denoted by $(a-1)\Pi^{a,b}_{n-1}$, have exactly one new neighbor, the corresponding vertex in the graph $(a-2)\Pi^{a,b}_{n-1}$. So, the contribution of this subgraph is $\Delta_{n-1,k-1}$. Similar analysis for $b$ copies of the graph $\Pi^{a,b}_{n-2}$ yields the contribution of $\Delta_{n-2,k-1}+(b-1)\Delta_{n-2,k-2}$. The subgraph $0\Pi^{a,b}_{n-1}$ needs to be treated differently, because it is the only copy neighboring two copies of different dimensions. Namely, it neighbors the copy  $1\Pi^{a,b}_{n-1}$ and the copy $0a\Pi^{a,b}_{n-2}$. This implies that some vertices in $0\Pi^{a,b}_{n-1}$ have two new neighbors and others have only one. More precisely, subgraph $0(a-1)\Pi^{a,b}_{n-2}$ of the subgraph $0\Pi^{a,b}_{n-1}$ contains vertices with two new neighbors because each vertex in $0(a-1)\Pi^{a,b}_{n-2}$ has one new neighbor in $0a\Pi^{a,b}_{n-2}$, and the other one in $0(a-2)\Pi^{a,b}_{n-2}\subset 0\Pi^{a,b}_{n-1}$. The remaining vertices in $0\Pi^{a,b}_{n-1}$ have one new neighbor. Vertices in subgraph $0(a-1)\Pi^{a,b}_{n-2}\subset\Pi^{a,b}_n$ have a degree for three greater than the corresponding vertices in the graph $\Pi^{a,b}_{n-2}$. Hence, the contribution of the subgraph $0\Pi^{a,b}_n$ to $\Delta_{n,k}$ is $\Delta_{n-1,k-1}-\Delta_{n-2,k-2}+\Delta_{n-2,k-3}$, and the claim follows.  
\end{proof}

Figure \ref{fig: Horadam_distribution_degree_Pi_n^{3,2}} shows the recursive nature of the degrees of vertices in the Horadam cube $\Pi^{3,2}_n$ for $n=1,2,3$. 

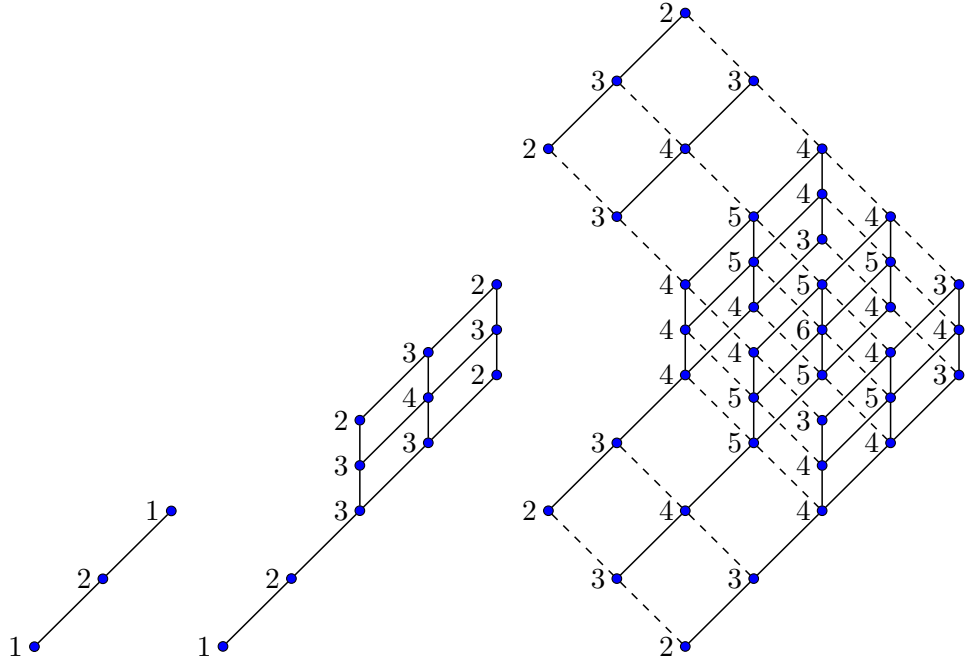
\begin{figure}[h!] \centering \begin{tikzpicture}[scale=0.6]
\tikzmath{\x1 = 0.35; \y1 =-0.05; \z1=180; \w1=0.2; \xs=-8; \ys=0; \yss=-5; \z2=90;
\x2 = \x1 + 1; \y2 =\y1 +3; } 
\small
\node [label={[label distance=\y1 cm]\z1: $1$},circle,fill=blue,draw=black,scale=\x1](A1) at (0,4) {};
\node [label={[label distance=\y1 cm]\z1: $2$},circle,fill=blue,draw=black,scale=\x1](A4) at (-1.5,2.5) {};
\node [label={[label distance=\y1 cm]\z1: $1$},circle,fill=blue,draw=black,scale=\x1](A7) at (-3,1) {};

\draw [line width=\w1 mm] (A1)--(A4)--(A7);

\end{tikzpicture} \centering \begin{tikzpicture}[scale=0.6]
\tikzmath{\x1 = 0.35; \y1 =-0.05; \z1=180; \z2=90; \w1=0.2; \xs=-8; \ys=0; \yss=-5;
\x2 = \x1 + 1; \y2 =\y1 +3; } 
\small
\node [label={[label distance=\y1 cm]\z1: $2$},circle,fill=blue,draw=black,scale=\x1](A1) at (0,4) {};
\node [label={[label distance=\y1 cm]\z1: $3$},circle,fill=blue,draw=black,scale=\x1](A2) at (0,5) {};
\node [label={[label distance=\y1 cm]\z1: $2$},circle,fill=blue,draw=black,scale=\x1](A3) at (0,6) {};
\node [label={[label distance=\y1 cm]\z1: $3$},circle,fill=blue,draw=black,scale=\x1](A4) at (-1.5,2.5) {};
\node [label={[label distance=\y1 cm]\z1: $4$},circle,fill=blue,draw=black,scale=\x1](A5) at (-1.5,3.5) {};
\node [label={[label distance=\y1 cm]\z1: $3$},circle,fill=blue,draw=black,scale=\x1](A6) at (-1.5,4.5) {};
\node [label={[label distance=\y1 cm]\z1: $3$},circle,fill=blue,draw=black,scale=\x1](A7) at (-3,1) {};
\node [label={[label distance=\y1 cm]\z1: $3$},circle,fill=blue,draw=black,scale=\x1](A8) at (-3,2) {};
\node [label={[label distance=\y1 cm]\z1: $2$},circle,fill=blue,draw=black,scale=\x1](A9) at (-3,3) {};
\node [label={[label distance=\y1 cm]\z1: $2$},circle,fill=blue,draw=black,scale=\x1](A28) at (-4.5,-0.5) {};
\node [label={[label distance=\y1 cm]\z1: $1$},circle,fill=blue,draw=black,scale=\x1](A37) at (-6,-2) {};

\draw [line width=\w1 mm] (A1)--(A2)--(A3) (A4)--(A5)--(A6) (A7)--(A8)--(A9)  (A1)--(A4)--(A7)--(A28)--(A37) (A2)--(A5)--(A8) (A3)--(A6)--(A9);

\end{tikzpicture} \begin{tikzpicture}[scale=0.6]
\tikzmath{\x1 = 0.35; \y1 =-0.05; \z1=180; \w1=0.2; \xs=-8; \ys=0; \yss=-5;
\x2 = \x1 + 1; \y2 =\y1 +3; } 
\small
\node [label={[label distance=\y1 cm]\z1: $3$},circle,fill=blue,draw=black,scale=\x1](A1) at (0,4) {};
\node [label={[label distance=\y1 cm]\z1: $4$},circle,fill=blue,draw=black,scale=\x1](A2) at (0,5) {};
\node [label={[label distance=\y1 cm]\z1: $4$},circle,fill=blue,draw=black,scale=\x1](A3) at (0,6) {};
\node [label={[label distance=\y1 cm]\z1: $4$},circle,fill=blue,draw=black,scale=\x1](A4) at (-1.5,2.5) {};
\node [label={[label distance=\y1 cm]\z1: $5$},circle,fill=blue,draw=black,scale=\x1](A5) at (-1.5,3.5) {};
\node [label={[label distance=\y1 cm]\z1: $5$},circle,fill=blue,draw=black,scale=\x1](A6) at (-1.5,4.5) {};
\node [label={[label distance=\y1 cm]\z1: $4$},circle,fill=blue,draw=black,scale=\x1](A7) at (-3,1) {};
\node [label={[label distance=\y1 cm]\z1: $4$},circle,fill=blue,draw=black,scale=\x1](A8) at (-3,2) {};
\node [label={[label distance=\y1 cm]\z1: $4$},circle,fill=blue,draw=black,scale=\x1](A9) at (-3,3) {};
\node [label={[label distance=\y1 cm]\z1: $4$},circle,fill=blue,draw=black,scale=\x1](A10) at (1.5,2.5) {};
\node [label={[label distance=\y1 cm]\z1: $5$},circle,fill=blue,draw=black,scale=\x1](A11) at (1.5,3.5) {};
\node [label={[label distance=\y1 cm]\z1: $4$},circle,fill=blue,draw=black,scale=\x1](A12) at (1.5,4.5) {};
\node [label={[label distance=\y1 cm]\z1: $5$},circle,fill=blue,draw=black,scale=\x1](A13) at (0,1) {};
\node [label={[label distance=\y1 cm]\z1: $6$},circle,fill=blue,draw=black,scale=\x1](A14) at (0,2) {};
\node [label={[label distance=\y1 cm]\z1: $5$},circle,fill=blue,draw=black,scale=\x1](A15) at (0,3) {};
\node [label={[label distance=\y1 cm]\z1: $5$},circle,fill=blue,draw=black,scale=\x1](A16) at (-1.5,-0.5) {};
\node [label={[label distance=\y1 cm]\z1: $5$},circle,fill=blue,draw=black,scale=\x1](A17) at (-1.5,0.5) {};
\node [label={[label distance=\y1 cm]\z1: $4$},circle,fill=blue,draw=black,scale=\x1](A18) at (-1.5,1.5) {};
\node [label={[label distance=\y1 cm]\z1: $3$},circle,fill=blue,draw=black,scale=\x1](A19) at (3,1) {};
\node [label={[label distance=\y1 cm]\z1: $4$},circle,fill=blue,draw=black,scale=\x1](A20) at (3,2) {};
\node [label={[label distance=\y1 cm]\z1: $3$},circle,fill=blue,draw=black,scale=\x1](A21) at (3,3) {};
\node [label={[label distance=\y1 cm]\z1: $4$},circle,fill=blue,draw=black,scale=\x1](A22) at (1.5,-0.5) {};
\node [label={[label distance=\y1 cm]\z1: $5$},circle,fill=blue,draw=black,scale=\x1](A23) at (1.5,0.5) {};
\node [label={[label distance=\y1 cm]\z1: $4$},circle,fill=blue,draw=black,scale=\x1](A24) at (1.5,1.5) {};
\node [label={[label distance=\y1 cm]\z1: $4$},circle,fill=blue,draw=black,scale=\x1](A25) at (0,-2) {};
\node [label={[label distance=\y1 cm]\z1: $4$},circle,fill=blue,draw=black,scale=\x1](A26) at (0,-1) {};
\node [label={[label distance=\y1 cm]\z1: $3$},circle,fill=blue,draw=black,scale=\x1](A27) at (0,0) {};
\node [label={[label distance=\y1 cm]\z1: $3$},circle,fill=blue,draw=black,scale=\x1](A28) at (-4.5,-0.5) {};
\node [label={[label distance=\y1 cm]\z1: $4$},circle,fill=blue,draw=black,scale=\x1](A29) at (-3,-2) {};
\node [label={[label distance=\y1 cm]\z1: $3$},circle,fill=blue,draw=black,scale=\x1](A30) at (-1.5,-3.5) {};
\node [label={[label distance=\y1 cm]\z1: $3$},circle,fill=blue,draw=black,scale=\x1](A31) at (-1.5,7.5) {};
\node [label={[label distance=\y1 cm]\z1: $4$},circle,fill=blue,draw=black,scale=\x1](A32) at (-3,6) {};
\node [label={[label distance=\y1 cm]\z1: $3$},circle,fill=blue,draw=black,scale=\x1](A33) at (-4.5,4.5) {};
\node [label={[label distance=\y1 cm]\z1: $2$},circle,fill=blue,draw=black,scale=\x1](A34) at (-3,9) {};
\node [label={[label distance=\y1 cm]\z1: $3$},circle,fill=blue,draw=black,scale=\x1](A35) at (-4.5,7.5) {};
\node [label={[label distance=\y1 cm]\z1: $2$},circle,fill=blue,draw=black,scale=\x1](A36) at (-6,6) {};
\node [label={[label distance=\y1 cm]\z1: $2$},circle,fill=blue,draw=black,scale=\x1](A37) at (-6,-2) {};
\node [label={[label distance=\y1 cm]\z1: $3$},circle,fill=blue,draw=black,scale=\x1](A38) at (-4.5,-3.5) {};
\node [label={[label distance=\y1 cm]\z1: $2$},circle,fill=blue,draw=black,scale=\x1](A39) at (-3,-5) {};

\draw [line width=\w1 mm] (A1)--(A2)--(A3) (A4)--(A5)--(A6) (A7)--(A8)--(A9) (A10)--(A11)--(A12) (A13)--(A14)--(A15) (A16)--(A17)--(A18) (A19)--(A20)--(A21) (A22)--(A23)--(A24) (A25)--(A26)--(A27) (A1)--(A4)--(A7)--(A28) (A2)--(A5)--(A8) (A3)--(A6)--(A9) (A10)--(A13)--(A16)--(A29) (A11)--(A14)--(A17) (A12)--(A15)--(A18) (A19)--(A22)--(A25)--(A30) (A20)--(A23)--(A26) (A21)--(A24)--(A27) (A31)--(A32)--(A33) (A34)--(A35)--(A36) (A28)--(A37) (A29)--(A38) (A30)--(A39); 

\draw [line width=\w1 mm,dashed, opacity=0.2] (A1)--(A10)--(A19) (A2)--(A11)--(A20) (A31)--(A3)--(A12)--(A21) (A4)--(A13)--(A22) (A5)--(A14)--(A23) (A32)--(A6)--(A15)--(A24) (A7)--(A16)--(A25) (A8)--(A17)--(A26) (A33)--(A9)--(A18)--(A27)  (A28)--(A29)--(A30) (A37)--(A38)--(A39) (A31)--(A34) (A32)--(A35) (A33)--(A36); 

\end{tikzpicture}
\caption{Recursive nature of the degrees in Horadam cube $\Pi^{3,2}_n$ for $n=1,2,3$.} \label{fig: Horadam_distribution_degree_Pi_n^{3,2}}
\end{figure}

\begin{theorem}\label{tm: Horadam_distribution of degrees a=1}
For $a=1$, the sequence $\Delta_{n,k}$ satisfies the recursive relation
\begin{align*}
\Delta_{n,k}&=\Delta_{n-1,k-1}+\Delta_{n-2,k-1}+(b-1)\Delta_{n-2,k-2}\\&+\Delta_{n-3,k-1}+(b-2)\Delta_{n-3,k-2}-(b-1)\Delta_{n-3,k-3},
\end{align*} with initial conditions $\Delta_{0,0}=1$, $\Delta_{1,1}=0$, $\Delta_{2,1}=2$ and $\Delta_{2,2}=b-2$. 
\end{theorem} \begin{proof}
By Theorem \ref{tm:candec_hor}, $\Pi^{1,b}_n$ contains one copy of the graph $\Pi^{1,b}_{n-1}$ and $b$ copies of the graph $\Pi^{1,b}_{n-2}$. Furthermore, by the same theorem, $\Pi^{1,b}_{n-1}$ can be further decomposed into one copy of $\Pi^{1,b}_{n-2}$ and $b$ copies of $\Pi^{1,b}_{n-3}$. So, $\Pi^{1,b}_{n}=P_{b+1}\square\Pi^{1,b}_{n-2}\oplus P_b\square \Pi^{1,b}_{n-3}$. Some vertices in the subgraph $\Pi^{1,b}_{n-1}=\Pi^{1,b}_{n-2}\oplus P_b\square\Pi^{1,b}_{n-3}$ have one new neighbor in $\Pi^{1,b}_{n}$, while others have no new neighbors. More precisely, vertices in subgraph $00l\Pi^{1,b}_{n-3}\subset\Pi^{1,b}_{n-1}$, where $1\leq l\leq b$, have no new neighbors, while vertices $00\Pi^{1,b}_{n-2}\subset \Pi^{1,b}_{n-1}$ have one new neighbor. Vertices in subgraph $00b\Pi^{1,b}_{n-3}\subset\Pi^{1,b}_n$ have a degree one greater than the corresponding vertices in the graph $\Pi^{1,b}_{n-3}$. Similarly, vertices in subgraph $00l\Pi^{1,b}_{n-3}\subset\Pi^{1,b}_n$, for $1\leq l\leq b-1$, have a degree for two greater than the corresponding vertices in the graph $\Pi^{1,b}_{n-3}$. So, the contribution of vertices in $0\Pi^{a,b}_{n-1}$ to $\Delta_{n,k}$ is $\Delta_{n-1,k-1}-\Delta_{n-3,k-2}+\Delta_{n-3,k-1}+(b-1)(-\Delta_{n-3,k-3}+\Delta_{n-3,k-2})$. What remains are $b$ copies of $\Pi^{1,b}_{n-2}$. Vertices in $0l\Pi^{1,b}_{n-2}$ contribute with $(b-1)\Delta_{n-2,k-2}$ for $1\leq l\leq b-1$, and with $\Delta_{n-2,k-1}$ for $l=b$. This completes the proof.
\end{proof}

From Theorems \ref{tm: Horadam_distribution of degrees a>=2} and \ref{tm: Horadam_distribution of degrees a=1} we can readily obtain bivariate generating functions. The generating function $\Delta(x,y)=\sum_{n,k}\Delta_{n,k}x^ny^k$ for the sequence $\Delta_{n,k}$ when $a\geq 2$ is  
\begin{align} \label{eq: Horadam_generating dunction_distribution of degrees a>=2}
\Delta(x,y)=\dfrac{1}{1-(2y+(a-2)y^2)x-(y+(b-2)y^2+y^3)x^2},
\end{align}
and for $a=1$ we have  
\begin{align}\label{eq: Horadam_generating dunction_distribution of degrees a=1}
\Delta(x,y)=\dfrac{1-(y-1)x}{1-xy-(y+(b-1)y^2)x^2-(y+(b-2)y^2-(b-1)y^3)x^3}.\end{align}

Now we present another way to derive the generating functions. Recall that $\mathcal{S}^{a,b}$ denotes the set of
all words from the alphabet $\left\lbrace 0,1,\dots,a+b-1\right\rbrace$ with the
property that the letters greater than or equal to $a$ can only appear immediately after $0$. Every word
starts with $0,\dots, a-1$, or with the block $0(a+l)$ for some $0\leq l\leq b-1$. Hence, 
$$\mathcal{S}^a=\varepsilon+0\mathcal{S}^{a,b}+1\mathcal{S}^{a,b}+\cdots+(a-1)\mathcal{S}^{a,b}+0a\mathcal{S}^{a,b}+\cdots+0(a+b-1)\mathcal{S}^{a,b},$$
where $\varepsilon$ denotes the empty word. Let
$\Delta(x,y)=\sum\limits_{n,k}\Delta_{n,k}x^ny^k$ be the formal power series
where $\Delta_{n,k}$ is the number of vertices of length $n$ having $k$
neighbors, as before. Consider the set that contains all vertices that start with $0$,
but not with the $0l$-block for $a\leq l\leq a+b-1$, and let $\Delta^0(x,y)$ denote the corresponding formal power
series. Then we have
\begin{align*}
\Delta(x,y)=1+\Delta^0(x,y)+\left((a-2)xy^2+xy+(b-1)x^2y^2 +x^2y\right)\Delta(x,y)\\
\Delta^0(x,y)=xy\left(1+\Delta^0(x,y)+\left((a-1)xy^2+(b-1)x^2y^2+x^2y\right)\Delta(x,y)\right).
\end{align*}  
Note that adding $0$ at the beginning of a vertex increases the number of
neighbors by one, except when the vertex starts with $a-1$, in which case
adding $0$ increases the number of neighbors by two. By solving the above
system of equations we readily obtain the generating function (\ref{eq: Horadam_generating dunction_distribution of degrees a>=2}).

In the case of $a=1$, adding $0$ at the beginning of a vertex increases the number of neighbors only if the vertex does not start with a $0l$ block. So we have
\begin{align*}
\Delta(x,y)=1+\Delta^0(x,y)+\left((b-1)x^2y^2 +x^2y\right)\Delta(x,y)\\
\Delta^0(x,y)=x\left(1+y\Delta^0(x,y)+\left((b-1)x^2y^2+x^2y\right)\Delta(x,y)\right),
\end{align*} 

which yields the generating function (\ref{eq: Horadam_generating dunction_distribution of degrees a=1}).

For the reader's convenience, we list the first few values of
$\Delta_{n,k}$ in Table \ref{table:Horadam_distribution of degrees}.
\begin{table}[h]\centering
$\begin{array}{ccccccc}
 a=1,b=2& a=2,b=2\\
\begin{array}{c|ccccccc}
n\backslash k & 1 & 2  & 3 & 4 & 5 \\ 
\hline
1 &  1 & 0 & 0 & 0 & 0  \\
2 &  2 & 1 & 0 & 0 & 0  \\
3 &  2 & 3 & 0 & 0 & 0 \\
4 &  1 & 4 & 5 & 1 & 0 \\
5 &  0 & 5 & 10& 6 & 0 \\
\end{array}&
\begin{array}{c|cccccccccc}
n\backslash k & 1 & 2  & 3 & 4 & 5 & 6 & 7 & 8\\ 
\hline
1 &  2 & 0 & 0 & 0 & 0 & 0 & 0 & 0 \\
2 & 1 & 4 & 1 & 0 & 0 & 0 & 0 & 0\\
3 &  0 & 4 & 8 & 4 & 0 & 0 & 0 & 0\\
4 & 0 & 1 & 12 & 18 & 12 & 1 & 0 & 0 \\
5 & 0 & 0 & 6 & 32 & 44 & 32 & 6 & 0 \\
\end{array}\\ [10ex]
\multicolumn{2}{c}{a=3,b=2}\\
\multicolumn{2}{c}{
\begin{array}{c|cccccccccc}
n\backslash k & 1 & 2  & 3 & 4 & 5 & 6 & 7 & 8 & 9 & 10\\ 
\hline
1 &  2 & 1 & 0 & 0 & 0 & 0 & 0 & 0 & 0 & 0\\
2 & 1 & 4 & 5 & 1 & 0 & 0 & 0 & 0 & 0 & 0\\
3 &  0 & 4 & 10 & 16 & 8 & 1 & 0 & 0& 0 & 0\\
4 & 0 & 1 & 12 & 30 & 47 & 37 & 11 & 1  & 0 & 0 \\
5 & 0 & 0 & 6 & 35 & 92 & 142 & 138 & 67 & 14 &  1 \\
\end{array}}
\end{array}$
\caption{The first values of $\Delta_{n,k}$ for $\Pi^{1,2}_n$, $\Pi^{2,2}_n$ and $\Pi^{3,2}_n$ .} \label{table:Horadam_distribution of degrees}
\end{table}

Also, they
do not (yet) appear in the {\em On-Line Encyclopedia of Integer Sequences}
\cite{oeis}.

\section{Subhypercubes} 
The {\em cube coefficient} of the graph $G$, denoted by $c_k(G)$, is the number of induced subhypercubes $Q_k$ in $G$. The {\em cube number} $c(G)$ is total number of induced subhypercubes in $G$, i.e., $c(G)=\sum_{k\geq 0}c_k(G)$. Observe that $c_0(G)=|V(G)|$ and $c_1(G)=|E(G)|$. Here we determine the cube coefficients $c_{k}$ for the Horadam cube $\Pi^{a,b}_n$.

\begin{theorem} \label{tm: Horadam_cube_coef}
Cube coefficients  $c_{k}(\Pi^{a,b}_n)$ for the Horadam cube $\Pi^{a,b}_n$ satisfy recursive relation \begin{align*}
c_k(\Pi^{a,b}_n)=ac_k(\Pi^{a,b}_{n-1})+bc_k(\Pi^{a,b}_{n-2})+(a-1)c_{k-1}(\Pi^{a,b}_{n-1})+bc_{k-1}(\Pi^{a,b}_{n-2}) 
\end{align*} with initial values $c_{0}(\Pi^{a,b}_n)=s^{a,b}_n$ and $c_{1}(\Pi^{a,b}_n)=a-1$. Furthermore, if $A(x,y)=\sum_{n\geq 0, k\geq 0}c_k(\Pi^{a,b}_n)x^ny^k$ denotes their  generating function, we have \begin{align*}A(x,y)=\dfrac{1}{1-ax-bx^2-(a-1)xy-bx^2y}. 
\end{align*}
\end{theorem} \begin{proof} By Theorem \ref{tm:candec_hor}, the Horadam cube $\Pi^{a,b}_n$ contains $a$ copies of the graph  $\Pi^{a,b}_{n-1}$ and $b$ copies of the graph $\Pi^{a,b}_{n-2}$. Then all induced subhypercubes in those copies contribute to the number of hypercubes of dimension $k$ in $\Pi^{a,b}_n$ with $ac_{k}(\Pi^{a,b}_{n-1})+bc_{k}(\Pi^{a,b}_{n-2})$. Moreover, the hypercube $Q_k=K_2\square Q_{k-1}$ can be part of the two adjacent copies of $\Pi^{a,b}_{n-1}$. Let $Q_{k-1}\subseteq\Pi^{a,b}_{n-1}$ be an induced subhypercube. Then, for $0\leq m\leq a-2$,  appending letters $m$ and $m+1$ to the vertices of $Q_{k-1}\subseteq\Pi^{a,b}_{n-1}$ yields a hypercube $Q_k$ in $\Pi^{a,b}_{n}$. Similar construction can be done for $Q_{k-1}\subseteq\Pi^{a,b}_{n-2}$. Hypercubes constructed as described contribute with $(a-1)c_{k-1}(\Pi^{a,b}_{n-1})+bc_{k-1}(\Pi^{a,b}_{n-2})$. From the recursive relation, one can easily obtain the generating function. 
\end{proof}

Disregarding the dimension of the subhypercube, Theorem \ref{tm: Horadam_cube_coef} gives a simple consequence.
\begin{corollary} Cube numbers $c(\Pi^{a,b}_n)$ for the Horadam cube $\Pi^{a,b}_n$ satisfy recursive relation \begin{align*}
c(\Pi^{a,b}_n)=(2a-1)c(\Pi^{a,b}_{n-1})+2bc(\Pi^{a,b}_{n-2})
\end{align*} with initial conditions $c(\Pi^{a,b}_0)=1$ and $c(\Pi^{a,b}_1)=2a-1$. If $A(x)=\sum_{n\geq 0}c(\Pi^{a,b}_n)x^n$ denote the  generating function, we have \begin{align*}A(x)=\dfrac{1}{1-(2a-1)x-2bx^2}.\end{align*} 
\end{corollary}

The cube polynomial was first introduced by Bre\v{s}ar, Klav\v{z}ar, and \v{S}krekovski \cite{CubePoly}. For a graph $G$ we define its {\em cube polynomial} as $$C(G,x)=\sum_{k\geq 0}c_k(G)x^k.$$  From the generating function $A(x,y)$, we have the following corollary that provides the cube polynomials for the Horadam cubes. 
\begin{corollary}
The cube polynomials $C(\Pi^{a,b}_n,x)$ satisfy the recursive relation \begin{align*}
C(\Pi^{a,b}_n,x)=(a+(a-1)x)C(\Pi^{a,b}_{n-1},x)+(b+bx)C(\Pi^{a,b}_{n-2},x)
\end{align*} with initial conditions $C(\Pi^{a,b}_0,x)=1$ and $C(\Pi^{a,b}_1,x)=a+(a-1)x$.
\end{corollary} 

Table \ref{table: Horadam_cube_poly} shows the first values of the cube polynomials for the Jacobsthal cube and $\Pi^{3,2}_n$.

\begin{table}[h]\centering
$\begin{array}{c|l|l}
n& c(\Pi^{1,2}_n,x) & c(\Pi^{3,2}_n,x)\\ 
\hline
0 &  1 &  1\\
1 &  1 &  2x+3\\
2 &  2x+3 & 4x^2+14x+11 \\
3 &  4x+5 & 8x^3+44 x^2+74x+39\\
4 &  4x^2+14x+11 &16x^4+120x^3+316 x^2+350x+139\\
5 &  12x^2+32x+21&32 x^5+304 x^4+1096 x^3+1884 x^2+1554 x+ 495\\
\end{array}$
\caption{Cube polynomials $C(\Pi^{a,b}_n,x)$ of the Horadam cubes.} \label{table: Horadam_cube_poly}
\end{table}

\section{Horadam cubes are (mostly) Hamiltonian}

At the beginning of this section, it is convenient to recall that any path that visits each vertex in the graph exactly once is called a {\em Hamiltonian path}, while a cycle with the same property is called a {\em Hamiltonian cycle}. A graph that contains a Hamiltonian cycle is {\em Hamiltonian}. In this section, we explore the possibility of extending the results on the existence of the Hamiltonian paths and  Hamiltonian cycles in metallic cubes to the Horadam cubes. The answer turns out to be positive. The following theorem extends the established result regarding the existence of Hamiltonian paths in the metallic cubes to encompass all Horadam cubes. 
\begin{theorem} \label{tm: Horadam_Hamiltonian_path} Let $a,b,n\geq 1$. Then the Horadam cube $\Pi^{a,b}_n$ contains a Hamiltonian path.
\end{theorem}

\begin{proof} Since, by Theorem \ref{tm:candec_hor}, Horadam cube admits canonical decomposition, i.e., $\Pi^{a,b}_n=P_a\square \Pi^{a,b}_{n-1}\oplus P_b\square \Pi^{a,b}_{n-2}$, it is only natural to construct a Hamiltonian path inductively. We claim that the Horadam cube $\Pi^{a,b}_n$ contains a Hamiltonian path with the following endpoints. In the case of $a$ odd and $b$ even, $\Pi^{a,b}_n$ contains Hamiltonian path with endpoints $0(a+b-1)0(a+b-1)\cdots 0(a+b-1)$ and $(a-1)0(a+b-1)0(a+b-1)\cdots (a+b-1)0$ for even $n$; and $0(a+b-1)0(a+b-1)\cdots 0(a+b-1)0$ and $(a-1)0(a+b-1)0(a+b-1)\cdots 0(a+b-1)$ for odd $n$. In the case of $a$ and $b$ odd, $\Pi^{a,b}_n$ contains Hamiltonian path with endpoints $0(a+b-1)(a-1)0(a+b-1)(a-1)\cdots 0(a+b-1)(a-1)$ and $(a-1)0(a+b-1)\cdots (a-1)0(a+b-1)$ if $n=3m$ for some $m\in \mathbb{N}$, $0(a+b-1)(a-1)0(a+b-1)(a-1)\cdots 0(a+b-1)$ and $(a-1)0(a+b-1)\cdots (a-1)0$ if $n=3m-1$, and $0(a+b-1)(a-1)0(a+b-1)(a-1)\cdots 0$ and $(a-1)0(a+b-1)\cdots (a-1)$ if $n=3m-2$. For $a$ and $b$ even endpoints are $(a-1)(a-1)\cdots (a-1)$ and $0(a+b-1)(a-1)\cdots (a-1)$. Finally, if $a$ is even and $b$ is odd we have endpoints $(a-1)(a-1)\cdots (a-1)$ and $0(a+b-1)\cdots 0(a+b-1)$ for $n$ even, and $(a-1)(a-1)\cdots (a-1)$ and $0(a+b-1)\cdots 0(a+b-1)0$ for $n$ odd. 

The base of induction is easily verified. For $n=1$, graph $\Pi^{a,b}_1$ is a path on $a$ vertices and the claim is true. For $n=2$, graph $\Pi^{a,b}_2$ is an $a\times a$ grid with the path that contains $b$ vertices appended to the vertex $0(a-1)$. Hence, we can construct a Hamiltonian path from the vertex $0(a+b-1)$ to the vertex $(a-1)0$ if $a$ is odd, or to the vertex $(a-1)(a-1)$ if $a$ is even. 

Suppose that for all $k<n$, $\Pi^{a,b}_k$ contains a Hamiltonian path with endpoints as described. First, consider the case of $a$ odd and $b$ even. If $n$ is even, $n-2$ is also even and $n-1$ is odd. By assumption, every subgraph $k\Pi^{a,b}_{n-1}$ and $0l\Pi^{a,b}_{n-2}$ contains a Hamiltonian path with endpoints as described. Since $a$ is odd, by adding the edges between corresponding endpoints of the Hamiltonian paths in the neighboring subgraphs $\Pi^{a,b}_{n-1}$, we can construct the Hamiltonian path in $P_a\square\Pi^{a,b}_{n-1}\subseteq \Pi^{a,b}_n$  with endpoints $(a-1)0(a+b-1)0(a+b-1)\cdots 0(a+b-1)0$ and $0(a-1)0(a+b-1)0(a+b-1)\cdots 0(a+b-1)$. In a similar way we construct a path in $P_b\square\Pi^{a,b}_{n-2}\subseteq \Pi^{a,b}_n$ with endpoints $0a0(a+b-1)0(a+b-1)\cdots 0(a+b-1)$ and $0(a+b-1)0(a+b-1)\cdots 0(a+b-1)$. Since the vertices $0(a-1)0(a+b-1)0(a+b-1)\cdots 0(a+b-1)$ and $0a0(a+b-1)0(a+b-1)\cdots 0(a+b-1)$ are adjacent in $\Pi^{a,b}_n$, we add the edge connecting them to obtain a Hamiltonian path in $\Pi^{a,b}_n$. If $n$ is odd, the path in $P_a\square\Pi^{a,b}_{n-1}\subseteq \Pi^{a,b}_n$ has endpoints $(a-1)0(a+b-1)0(a+b-1)\cdots 0(a+b-1)$ and $0(a-1)0(a+b-1)0(a+b-1)\cdots 0(a+b-1)0$, the path in $P_b\square\Pi^{a,b}_{n-2}\subseteq \Pi^{a,b}_n$ has endpoints $0a0(a+b-1)0(a+b-1)\cdots 0(a+b-1)0$ and $0(a+b-1)0(a+b-1)\cdots 0(a+b-1)0$. Again, since the vertices $0a0(a+b-1)0(a+b-1)\cdots 0(a+b-1)0$ and $0(a-1)0(a+b-1)0(a+b-1)\cdots 0(a+b-1)0$ are adjacent, we add the edge connecting them to obtain a Hamiltonian path in $\Pi^{a,b}_n$.  

In the second case, if $n=3m$, the Hamiltonian path in $P_a\square\Pi^{a,b}_{n-1}\subseteq \Pi^{a,b}_n$ has endpoints $(a-1)0(a+b-1)(a-1)0(a+b-1)(a-1)\cdots 0(a+b-1)$ and $0(a-1)0(a+b-1)\cdots (a-1)0$, while the Hamiltonian path in $P_b\square\Pi^{a,b}_{n-2}\subseteq \Pi^{a,b}_n$ with endpoints $0a0(a+b-1)(a-1)0(a+b-1)(a-1)\cdots (a-1)0$ and $0(a+b-1)(a-1)0(a+b-1)\cdots (a-1)$. Since the vertices $0(a-1)0(a+b-1)\cdots (a-1)0$ and $0a0(a+b-1)(a-1)0(a+b-1)(a-1)\cdots (a-1)0$ are adjacent in $\Pi^{a,b}_n$, we add the edge connecting them to obtain a Hamiltonian path in $\Pi^{a,b}_n$. If $n=3m-1$, we remove the last letter in each endpoint, and the argument is the same. Similarly, if $n=3m-2$, we remove two last letters. 

Now we consider the case where $a$ and $b$ are even. The constructed Hamiltonian path in $P_a\square\Pi^{a,b}_{n-1}\subseteq \Pi^{a,b}_n$ has endpoints $(a-1)(a-1)\cdots (a-1)$ and $0(a-1)\cdots (a-1)$, the Hamiltonian path in $P_b\square\Pi^{a,b}_{n-2}\subseteq \Pi^{a,b}_n$ has endpoints $0a(a-1)\cdots (a-1)$ and $0(a+b-1)(a-1)\cdots (a-1)$. Again, vertices $0(a-1)\cdots (a-1)$ and $0a(a-1)\cdots (a-1)$ are adjacent in $\Pi^{a,b}_n$, we obtained a Hamiltonian path in $\Pi^{a,b}_n$. 

Finally, the last case contains $a$ even and $b$ odd. In this case, the Hamiltonian path in $P_a\square\Pi^{a,b}_{n-1}\subseteq \Pi^{a,b}_n$ has endpoints $(a-1)(a-1)\cdots (a-1)$ and $0(a-1)\cdots (a-1)$, the Hamiltonian path in $P_b\square\Pi^{a,b}_{n-2}\subseteq \Pi^{a,b}_n$ has endpoints $0a(a-1)\cdots (a-1)$ and $0(a+b-1)0(a+b-1)\cdots 0(a+b-1)$ for even $n$ and $0(a+b-1)0(a+b-1)\cdots 0$ for odd $n$. Vertices $0(a-1)\cdots (a-1)$ and $0a(a-1)\cdots (a-1)$ are adjacent in $\Pi^{a,b}_n$, so we have a Hamiltonian path in $\Pi^{a,b}_n$. This last case is demonstrated in Figure \ref{fig: Hamil_path_in_Pi_4^{2,3}}. 

\end{proof}

\begin{figure}[h!] \centering \begin{tikzpicture}[scale=1]
\tikzmath{\x1 = 0.35; \y1 =-0.05; \z1=180; \w1=0.2; \xs=-115; \ys=-30; \yss=-5; \t=-0.5; \u=-0.5;
\x2 = \x1 + 1; \y2 =\y1 +3; } 
\scriptsize

\node [label={[label distance=\y1 cm]\z1: $0000$},circle,fill=blue,draw=black,scale=\x1](A1) at (1,2) {};
\node [label={[label distance=\y1 cm]\z1: $0001$},circle,fill=blue,draw=black,scale=\x1](A2) at (0,1) {};
\node [label={[label distance=\y1 cm]\z1: $0101$},circle,fill=blue,draw=black,scale=\x1](A3) at (1,0) {};
\node [label={[label distance=\y1 cm]\z1: $0100$},circle,fill=blue,draw=black,scale=\x1](A4) at (2,1) {};
\node [label={[label distance=\y1 cm]\z1: $0010$},circle,fill=blue,draw=black,scale=\x1](A5) at (1,3) {};
\node [label={[label distance=\y1 cm]\z1: $0011$},circle,fill=blue,draw=black,scale=\x1](A6) at (0,2) {};
\node [label={[label distance=\y1 cm]\z1: $0111$},circle,fill=blue,draw=black,scale=\x1](A7) at (1,1) {};
\node [label={[label distance=\y1 cm]\z1: $0110$},circle,fill=blue,draw=black,scale=\x1](A8) at (2,2) {};
\node [label={[label distance=\y1 cm]\z1: $0002$},circle,fill=blue,draw=black,scale=\x1](A9) at (\t,\t+1) {};
\node [label={[label distance=\y1 cm]\z1: $0003$},circle,fill=blue,draw=black,scale=\x1](A10) at (2*\t,2*\t+1) {};
\node [label={[label distance=\y1 cm]\z1: $0102$},circle,fill=blue,draw=black,scale=\x1](A11) at (\u+1,\u) {};
\node [label={[label distance=\y1 cm]\z1: $0103$},circle,fill=blue,draw=black,scale=\x1](A12) at (2*\u+1,2*\u) {};
\node [label={[label distance=\y1 cm]\z1: $0021$},circle,fill=blue,draw=black,scale=\x1](A13) at (\t,-\t+2) {};
\node [label={[label distance=\y1 cm]\z1: $0031$},circle,fill=blue,draw=black,scale=\x1](A14) at (2*\t,-2*\t+2) {};
\node [label={[label distance=\y1 cm]\z1: $0020$},circle,fill=blue,draw=black,scale=\x1](A15) at (\t+1,-\t+3) {};
\node [label={[label distance=\y1 cm]\z1: $0030$},circle,fill=blue,draw=black,scale=\x1](A16) at (2*\t+1,-2*\t+3) {};
\node [label={[label distance=\y1 cm]\z1: $0040$},circle,fill=blue,draw=black,scale=\x1](A16a) at (3*\t+1,-3*\t+3) {};
\node [label={[label distance=\y1 cm]\z1: $0041$},circle,fill=blue,draw=black,scale=\x1](A16b) at (3*\t,-3*\t+2) {};
\node [label={[label distance=\y1 cm]\z1: $0004$},circle,fill=blue,draw=black,scale=\x1](A16c) at (3*\t,3*\t+1) {};
\node [label={[label distance=\y1 cm]\z1: $0104$},circle,fill=blue,draw=black,scale=\x1](A16d) at (3*\u+1,3*\u) {};

\node [label={[label distance=\y1 cm]\z1: $1000$},xshift=\xs,yshift=\ys,circle,fill=blue,draw=black,scale=\x1](A17) at (1,2) {};
\node [label={[label distance=\y1 cm]\z1: $1001$},xshift=\xs,yshift=\ys,circle,fill=blue,draw=black,scale=\x1](A18) at (0,1) {};
\node [label={[label distance=\y1 cm]\z1: $1101$},xshift=\xs,yshift=\ys,circle,fill=blue,draw=black,scale=\x1](A19) at (1,0) {};
\node [label={[label distance=\y1 cm]\z1: $1100$},xshift=\xs,yshift=\ys,circle,fill=blue,draw=black,scale=\x1](A20) at (2,1) {};
\node [label={[label distance=\y1 cm]\z1: $1010$},xshift=\xs,yshift=\ys,circle,fill=blue,draw=black,scale=\x1](A21) at (1,3) {};
\node [label={[label distance=\y1 cm]\z1: $1011$},xshift=\xs,yshift=\ys,circle,fill=blue,draw=black,scale=\x1](A22) at (0,2) {};
\node [label={[label distance=\y1 cm]\z1: $1111$},xshift=\xs,yshift=\ys,circle,fill=blue,draw=black,scale=\x1](A23) at (1,1) {};
\node [label={[label distance=\y1 cm]\z1: $1110$},xshift=\xs,yshift=\ys,circle,fill=blue,draw=black,scale=\x1](A24) at (2,2) {};
\node [label={[label distance=\y1 cm]\z1: $1002$},xshift=\xs,yshift=\ys,circle,fill=blue,draw=black,scale=\x1](A25) at (\t,\t+1) {};
\node [label={[label distance=\y1 cm]\z1: $1003$},xshift=\xs,yshift=\ys,circle,fill=blue,draw=black,scale=\x1](A26) at (2*\t,2*\t+1) {};
\node [label={[label distance=\y1 cm]\z1: $1102$},xshift=\xs,yshift=\ys,circle,fill=blue,draw=black,scale=\x1](A27) at (\u+1,\u) {};
\node [label={[label distance=\y1 cm]\z1: $1103$},xshift=\xs,yshift=\ys,circle,fill=blue,draw=black,scale=\x1](A28) at (2*\u+1,2*\u) {};
\node [label={[label distance=\y1 cm]\z1: $1021$},xshift=\xs,yshift=\ys,circle,fill=blue,draw=black,scale=\x1](A29) at (\t,-\t+2) {};
\node [label={[label distance=\y1 cm]\z1: $1031$},xshift=\xs,yshift=\ys,circle,fill=blue,draw=black,scale=\x1](A30) at (2*\t,-2*\t+2) {};
\node [label={[label distance=\y1 cm]\z1: $1020$},xshift=\xs,yshift=\ys,circle,fill=blue,draw=black,scale=\x1](A31) at (\t+1,-\t+3) {};
\node [label={[label distance=\y1 cm]\z1: $1030$},xshift=\xs,yshift=\ys,circle,fill=blue,draw=black,scale=\x1](A32) at (2*\t+1,-2*\t+3) {};
\node [label={[label distance=\y1 cm]\z1: $1040$},xshift=\xs,yshift=\ys,circle,fill=blue,draw=black,scale=\x1](A32a) at (3*\t+1,-3*\t+3) {};
\node [label={[label distance=\y1 cm]\z1: $1041$},xshift=\xs,yshift=\ys,circle,fill=blue,draw=black,scale=\x1](A32b) at (3*\t,-3*\t+2) {};
\node [label={[label distance=\y1 cm]\z1: $1004$},xshift=\xs,yshift=\ys,circle,fill=blue,draw=black,scale=\x1](A32c) at (3*\t,3*\t+1) {};
\node [label={[label distance=\y1 cm]\z1: $1104$},xshift=\xs,yshift=\ys,circle,fill=blue,draw=black,scale=\x1](A32d) at (3*\u+1,3*\u) {};

\node [label={[label distance=\y1 cm]\z1: $0201$},xshift=-0.5*\xs,yshift=-0.5*\ys,circle,fill=blue,draw=black,scale=\x1](A33) at (1,0) {};
\node [label={[label distance=\y1 cm]\z1: $0200$},xshift=-0.5*\xs,yshift=-0.5*\ys,circle,fill=blue,draw=black,scale=\x1](A34) at (2,1) {};
\node [label={[label distance=\y1 cm]\z1: $0211$},xshift=-0.5*\xs,yshift=-0.5*\ys,circle,fill=blue,draw=black,scale=\x1](A35) at (1,1) {};
\node [label={[label distance=\y1 cm]\z1: $0210$},xshift=-0.5*\xs,yshift=-0.5*\ys,circle,fill=blue,draw=black,scale=\x1](A36) at (2,2) {};
\node [label={[label distance=\y1 cm]\z1: $0202$},xshift=-0.5*\xs,yshift=-0.5*\ys,circle,fill=blue,draw=black,scale=\x1](A37) at (\u+1,\u) {};
\node [label={[label distance=\y1 cm]\z1: $0203$},xshift=-0.5*\xs,yshift=-0.5*\ys,circle,fill=blue,draw=black,scale=\x1](A38) at (2*\u+1,2*\u) {};
\node [label={[label distance=\y1 cm]\z1: $0204$},xshift=-0.5*\xs,yshift=-0.5*\ys,circle,fill=blue,draw=black,scale=\x1](A38a) at (3*\u+1,3*\u) {};

\node [label={[label distance=\y1 cm]\z1: $0301$},xshift=-\xs,yshift=-\ys,circle,fill=blue,draw=black,scale=\x1](A39) at (1,0) {};
\node [label={[label distance=\y1 cm]\z1: $0300$},xshift=-\xs,yshift=-\ys,circle,fill=blue,draw=black,scale=\x1](A40) at (2,1) {};
\node [label={[label distance=\y1 cm]\z1: $0311$},xshift=-\xs,yshift=-\ys,circle,fill=blue,draw=black,scale=\x1](A41) at (1,1) {};
\node [label={[label distance=\y1 cm]\z1: $0310$},xshift=-\xs,yshift=-\ys,circle,fill=blue,draw=black,scale=\x1](A42) at (2,2) {};
\node [label={[label distance=\y1 cm]\z1: $0302$},xshift=-\xs,yshift=-\ys,circle,fill=blue,draw=black,scale=\x1](A43) at (\u+1,\u) {};
\node [label={[label distance=\y1 cm]\z1: $0303$},xshift=-\xs,yshift=-\ys,circle,fill=blue,draw=black,scale=\x1](A44) at (2*\u+1,2*\u) {};
\node [label={[label distance=\y1 cm]\z1: $0304$},xshift=-\xs,yshift=-\ys,circle,fill=blue,draw=black,scale=\x1](A44a) at (3*\u+1,3*\u) {};

\node [label={[label distance=\y1 cm]\z1: $0401$},xshift=-1.5*\xs,yshift=-1.5*\ys,circle,fill=blue,draw=black,scale=\x1](A45) at (1,0) {};
\node [label={[label distance=\y1 cm]\z1: $0400$},xshift=-1.5*\xs,yshift=-1.5*\ys,circle,fill=blue,draw=black,scale=\x1](A46) at (2,1) {};
\node [label={[label distance=\y1 cm]\z1: $0411$},xshift=-1.5*\xs,yshift=-1.5*\ys,circle,fill=blue,draw=black,scale=\x1](A47) at (1,1) {};
\node [label={[label distance=\y1 cm]\z1: $0410$},xshift=-1.5*\xs,yshift=-1.5*\ys,circle,fill=blue,draw=black,scale=\x1](A48) at (2,2) {};
\node [label={[label distance=\y1 cm]\z1: $0402$},xshift=-1.5*\xs,yshift=-1.5*\ys,circle,fill=blue,draw=black,scale=\x1](A49) at (\u+1,\u) {};
\node [label={[label distance=\y1 cm]\z1: $0403$},xshift=-1.5*\xs,yshift=-1.5*\ys,circle,fill=blue,draw=black,scale=\x1](A50) at (2*\u+1,2*\u) {};
\node [label={[label distance=\y1 cm]\z1: $0404$},xshift=-1.5*\xs,yshift=-1.5*\ys,circle,fill=blue,draw=black,scale=\x1](A51) at (3*\u+1,3*\u) {};

\draw [line width=\w1 mm,dashed, opacity=0.2] (A1)--(A17) (A2)--(A18) (A19)--(A3)--(A33)--(A39)--(A45) (A20)--(A4)--(A34)--(A40)--(A46)  (A5)--(A21)  (A6)--(A22) (A23)--(A7)--(A35)--(A41)  (A24)--(A8)--(A36)--(A42)--(A48)   (A25)--(A9) (A26)--(A10) (A27)--(A11)--(A37)--(A43)--(A49)  (A28)--(A12)--(A38)--(A44)--(A50)  (A13)--(A29) (A14)--(A30) (A15)--(A31) (A16)--(A32) (A1)--(A2)--(A3)--(A4)--(A1)--(A5)--(A6)--(A7)--(A8)--(A5)--(A15)--(A16)--(A14)--(A13)--(A6) (A2)--(A9)--(A10)--(A12)--(A11)--(A3) (A2)--(A6) (A3)--(A7)  (A4)--(A8)  (A9)--(A11) (A13)--(A15)  (A17)--(A18)--(A19)--(A20)--(A17)--(A21)--(A22)--(A23)--(A24)--(A21)--(A31)--(A32)--(A30)--(A29)--(A22) (A18)--(A25)--(A26)--(A32c)--(A32d)--(A28)--(A27)--(A19) (A18)--(A22) (A19)--(A23)  (A20)--(A24)  (A25)--(A27) (A29)--(A31)  (A33)--(A34)--(A36)--(A35)--(A33)--(A37)--(A38)--(A38a) (A39)--(A40)--(A42)--(A41)--(A39)--(A43)--(A44)--(A44a)  (A47)--(A45) (A26)--(A28) (A32)--(A32a)--(A32b)--(A30) (A16)--(A16a)--(A16b)--(A14)  (A10)--(A16c)--(A16d)--(A12) (A16a)--(A32a) (A16b)--(A32b) (A16c)--(A32c) (A51)--(A44a)--(A38a)--(A16d)--(A32d);

\draw [line width=\w1 mm] (A51)--(A50)--(A49)--(A45)--(A46)--(A48)--(A47)--(A41)--(A42)--(A40)--(A39)--(A43)--(A44)--(A44a)--(A38a)--(A38)--(A37)--(A33)--(A34)--(A36)--(A35)--(A7)--(A8)--(A4)--(A3)--(A11)--(A12)--(A16d)--(A16c)--(A10)--(A9)--(A2)--(A1)--(A5)--(A6)--(A13)--(A15)--(A16)--(A14)--(A16b)--(A16a)--(A32a)--(A32b)--(A30)--(A32)--(A31)--(A29)--(A22)--(A21)--(A17)--(A18)--(A25)--(A26)--(A32c)--(A32d)--(A28)--(A27)--(A19)--(A20)--(A24)--(A23) ; 

\end{tikzpicture}  
\caption{The Hamiltonian path in the Horadam cube $\Pi_4^{2,3}$.} \label{fig: Hamil_path_in_Pi_4^{2,3}}
\end{figure}
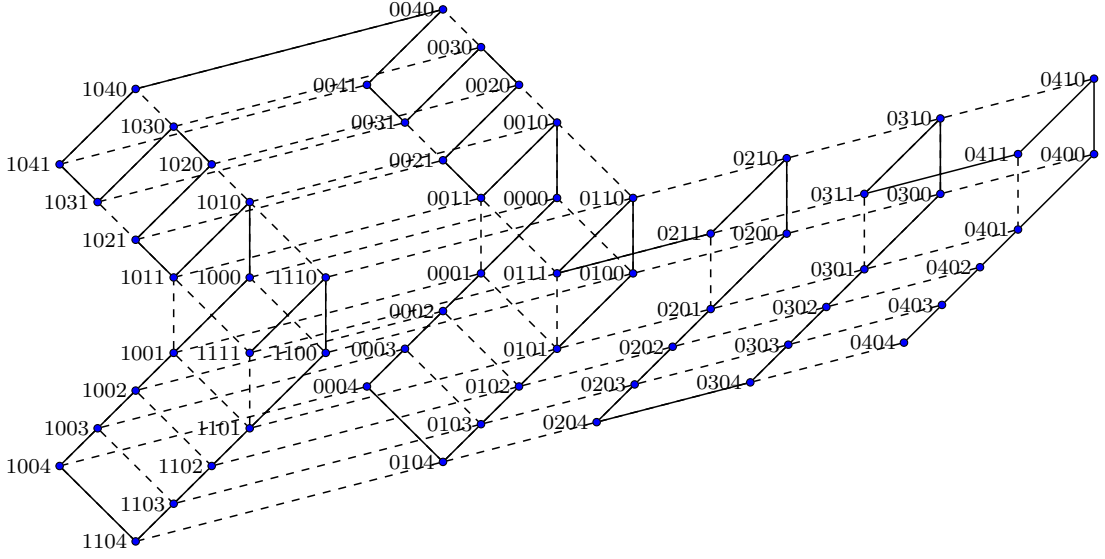

Metallic cubes are Hamiltonian if $a$ is even and $n$ is odd \cite{metallic}. The next few theorems provide the necessary conditions for the Horadam cube to be Hamiltonian.   

\begin{theorem} Let $a$ and $b$ be even. Then the Horadam cube $\Pi^{a,b}_n$ is Hamiltonian for every $n\geq 3$.\label{tm: Hor_Ham_cycle_a_even_b_even}
\end{theorem} \begin{proof} 
By Theorem \ref{tm:candec_hor}, the Horadam cube $\Pi^{a,b}_n$ admit canonical decomposition, i.e.,  $\Pi^{a,b}_n=P_a\square \Pi^{a,b}_{n-1}\oplus P_b\square \Pi^{a,b}_{n-2}$. There are even number of copies of the graphs $\Pi^{a,b}_{n-1}$ and $\Pi^{a,b}_{n-2}$ in $\Pi^{a,b}_n$, and, by Theorem \ref{tm: Horadam_Hamiltonian_path}, each copy contains a Hamiltonian path. Consider the subgraphs $0\Pi^{a,b}_{n-1}$ and $1\Pi^{a,b}_{n-1}$. Similar to the case of metallic cubes, we join the endpoints of  Hamiltonian paths in those subgraphs to obtain a Hamiltonian cycle in the subgraph $0\Pi^{a,b}_{n-1}\oplus 1\Pi^{a,b}_{n-1}$. A similar construction can be done for the remaining copies of subgraphs $\Pi^{a,b}_{n-1}$ as well as for the $b$ copies of the subgraph $\Pi^{a,b}_{n-2}$. Note that if some edge is part of a Hamiltonian cycle in one copy of the subgraph $\Pi^{a,b}_{n-1}$, then it is part of a Hamiltonian cycle in every copy of $\Pi^{a,b}_{n-1}$. The same observation holds for the subgraphs $\Pi^{a,b}_{n-2}$. Also, note that if the edge $(0a\gamma)(0a\delta)$ lies in the cycle in $0a\Pi^{a,b}_{n-2}$ for some $\gamma,\delta \in\mathcal{S}^{a,b}_{n-2}$, then the edge $(0(a-1)\gamma)(0(a-1)\delta)$ belongs to the cycle in $0\Pi^{a,b}_{n-1}$. The latter property holds because the Hamiltonian paths are constructed inductively, so the Hamiltonian paths in $0(a-1)\Pi^{a,b}_{n-2}$ and $0a\Pi^{a,b}_{n-2}$ are constructed in the same way. We construct a Hamiltonian cycle in $\Pi^{a,b}_n$ in following way: if the edge $(1\alpha)(1\beta)$ is part of a Hamiltonian cycle of the subgraph $1\Pi^{a,b}_{n-1}$ for some $\alpha,\beta\in\mathcal{S}^{a,b}_{n-1}$, then the edge  $(2\alpha)(2\beta)$ is part of a Hamiltonian cycle of the subgraph $2\Pi^{a,b}_{n-1}$. Removing the edges $(1\alpha)(1\beta)$ and $(2\alpha)(2\beta)$, and adding new edges $(0\alpha)(1\alpha)$ and $(0\beta)(1\beta)$ yields a Hamiltonian cycle in $0\Pi^{a,b}_{n-1}\oplus1\Pi^{a,b}_{n-1}\oplus2\Pi^{a,b}_{n-1}\oplus3\Pi^{a,b}_{n-1}$. Since the numbers $a$ and $b$ are even, it is clear that this method can be further extended to obtain a Hamiltonian cycle of the $\Pi^{a,b}_n$.
\end{proof} As an example,  Figure \ref{fig:Hamil_cycle_in_Pi_4^{2,2}} shows the Hamiltonian cycle in the Horadam cube $\Pi^{2,2}_4$. 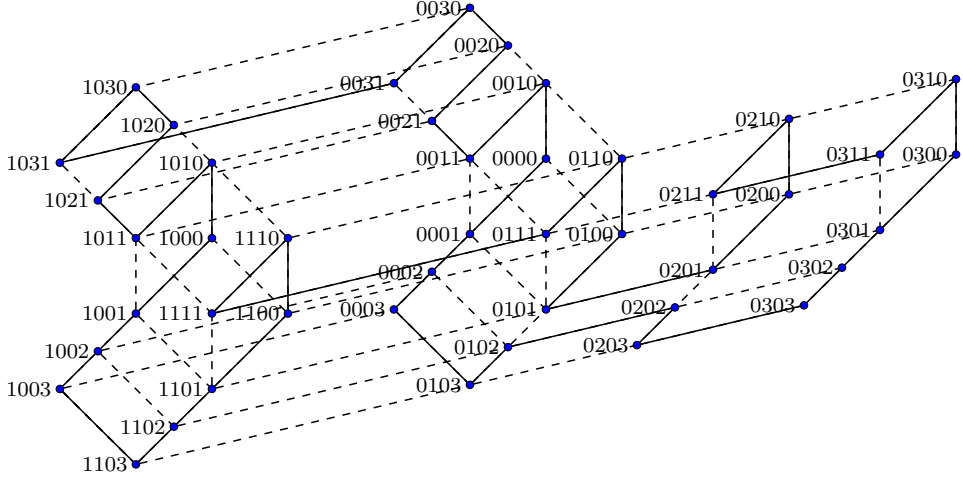
\begin{figure}[h!] \centering \begin{tikzpicture}[scale=1]
\tikzmath{\x1 = 0.35; \y1 =-0.05; \z1=180; \w1=0.2; \xs=-125; \ys=-30; \yss=-5; \t=-0.5; \u=-0.5;
\x2 = \x1 + 1; \y2 =\y1 +3; } 
\scriptsize

\node [label={[label distance=\y1 cm]\z1: $0000$},circle,fill=blue,draw=black,scale=\x1](A1) at (1,2) {};
\node [label={[label distance=\y1 cm]\z1: $0001$},circle,fill=blue,draw=black,scale=\x1](A2) at (0,1) {};
\node [label={[label distance=\y1 cm]\z1: $0101$},circle,fill=blue,draw=black,scale=\x1](A3) at (1,0) {};
\node [label={[label distance=\y1 cm]\z1: $0100$},circle,fill=blue,draw=black,scale=\x1](A4) at (2,1) {};
\node [label={[label distance=\y1 cm]\z1: $0010$},circle,fill=blue,draw=black,scale=\x1](A5) at (1,3) {};
\node [label={[label distance=\y1 cm]\z1: $0011$},circle,fill=blue,draw=black,scale=\x1](A6) at (0,2) {};
\node [label={[label distance=\y1 cm]\z1: $0111$},circle,fill=blue,draw=black,scale=\x1](A7) at (1,1) {};
\node [label={[label distance=\y1 cm]\z1: $0110$},circle,fill=blue,draw=black,scale=\x1](A8) at (2,2) {};
\node [label={[label distance=\y1 cm]\z1: $0002$},circle,fill=blue,draw=black,scale=\x1](A9) at (\t,\t+1) {};
\node [label={[label distance=\y1 cm]\z1: $0003$},circle,fill=blue,draw=black,scale=\x1](A10) at (2*\t,2*\t+1) {};
\node [label={[label distance=\y1 cm]\z1: $0102$},circle,fill=blue,draw=black,scale=\x1](A11) at (\u+1,\u) {};
\node [label={[label distance=\y1 cm]\z1: $0103$},circle,fill=blue,draw=black,scale=\x1](A12) at (2*\u+1,2*\u) {};
\node [label={[label distance=\y1 cm]\z1: $0021$},circle,fill=blue,draw=black,scale=\x1](A13) at (\t,-\t+2) {};
\node [label={[label distance=\y1 cm]\z1: $0031$},circle,fill=blue,draw=black,scale=\x1](A14) at (2*\t,-2*\t+2) {};
\node [label={[label distance=\y1 cm]\z1: $0020$},circle,fill=blue,draw=black,scale=\x1](A15) at (\t+1,-\t+3) {};
\node [label={[label distance=\y1 cm]\z1: $0030$},circle,fill=blue,draw=black,scale=\x1](A16) at (2*\t+1,-2*\t+3) {};

\node [label={[label distance=\y1 cm]\z1: $1000$},xshift=\xs,yshift=\ys,circle,fill=blue,draw=black,scale=\x1](A17) at (1,2) {};
\node [label={[label distance=\y1 cm]\z1: $1001$},xshift=\xs,yshift=\ys,circle,fill=blue,draw=black,scale=\x1](A18) at (0,1) {};
\node [label={[label distance=\y1 cm]\z1: $1101$},xshift=\xs,yshift=\ys,circle,fill=blue,draw=black,scale=\x1](A19) at (1,0) {};
\node [label={[label distance=\y1 cm]\z1: $1100$},xshift=\xs,yshift=\ys,circle,fill=blue,draw=black,scale=\x1](A20) at (2,1) {};
\node [label={[label distance=\y1 cm]\z1: $1010$},xshift=\xs,yshift=\ys,circle,fill=blue,draw=black,scale=\x1](A21) at (1,3) {};
\node [label={[label distance=\y1 cm]\z1: $1011$},xshift=\xs,yshift=\ys,circle,fill=blue,draw=black,scale=\x1](A22) at (0,2) {};
\node [label={[label distance=\y1 cm]\z1: $1111$},xshift=\xs,yshift=\ys,circle,fill=blue,draw=black,scale=\x1](A23) at (1,1) {};
\node [label={[label distance=\y1 cm]\z1: $1110$},xshift=\xs,yshift=\ys,circle,fill=blue,draw=black,scale=\x1](A24) at (2,2) {};
\node [label={[label distance=\y1 cm]\z1: $1002$},xshift=\xs,yshift=\ys,circle,fill=blue,draw=black,scale=\x1](A25) at (\t,\t+1) {};
\node [label={[label distance=\y1 cm]\z1: $1003$},xshift=\xs,yshift=\ys,circle,fill=blue,draw=black,scale=\x1](A26) at (2*\t,2*\t+1) {};
\node [label={[label distance=\y1 cm]\z1: $1102$},xshift=\xs,yshift=\ys,circle,fill=blue,draw=black,scale=\x1](A27) at (\u+1,\u) {};
\node [label={[label distance=\y1 cm]\z1: $1103$},xshift=\xs,yshift=\ys,circle,fill=blue,draw=black,scale=\x1](A28) at (2*\u+1,2*\u) {};
\node [label={[label distance=\y1 cm]\z1: $1021$},xshift=\xs,yshift=\ys,circle,fill=blue,draw=black,scale=\x1](A29) at (\t,-\t+2) {};
\node [label={[label distance=\y1 cm]\z1: $1031$},xshift=\xs,yshift=\ys,circle,fill=blue,draw=black,scale=\x1](A30) at (2*\t,-2*\t+2) {};
\node [label={[label distance=\y1 cm]\z1: $1020$},xshift=\xs,yshift=\ys,circle,fill=blue,draw=black,scale=\x1](A31) at (\t+1,-\t+3) {};
\node [label={[label distance=\y1 cm]\z1: $1030$},xshift=\xs,yshift=\ys,circle,fill=blue,draw=black,scale=\x1](A32) at (2*\t+1,-2*\t+3) {};

\node [label={[label distance=\y1 cm]\z1: $0201$},xshift=-0.5*\xs,yshift=-0.5*\ys,circle,fill=blue,draw=black,scale=\x1](A33) at (1,0) {};
\node [label={[label distance=\y1 cm]\z1: $0200$},xshift=-0.5*\xs,yshift=-0.5*\ys,circle,fill=blue,draw=black,scale=\x1](A34) at (2,1) {};
\node [label={[label distance=\y1 cm]\z1: $0211$},xshift=-0.5*\xs,yshift=-0.5*\ys,circle,fill=blue,draw=black,scale=\x1](A35) at (1,1) {};
\node [label={[label distance=\y1 cm]\z1: $0210$},xshift=-0.5*\xs,yshift=-0.5*\ys,circle,fill=blue,draw=black,scale=\x1](A36) at (2,2) {};
\node [label={[label distance=\y1 cm]\z1: $0202$},xshift=-0.5*\xs,yshift=-0.5*\ys,circle,fill=blue,draw=black,scale=\x1](A37) at (\u+1,\u) {};
\node [label={[label distance=\y1 cm]\z1: $0203$},xshift=-0.5*\xs,yshift=-0.5*\ys,circle,fill=blue,draw=black,scale=\x1](A38) at (2*\u+1,2*\u) {};

\node [label={[label distance=\y1 cm]\z1: $0301$},xshift=-\xs,yshift=-\ys,circle,fill=blue,draw=black,scale=\x1](A39) at (1,0) {};
\node [label={[label distance=\y1 cm]\z1: $0300$},xshift=-\xs,yshift=-\ys,circle,fill=blue,draw=black,scale=\x1](A40) at (2,1) {};
\node [label={[label distance=\y1 cm]\z1: $0311$},xshift=-\xs,yshift=-\ys,circle,fill=blue,draw=black,scale=\x1](A41) at (1,1) {};
\node [label={[label distance=\y1 cm]\z1: $0310$},xshift=-\xs,yshift=-\ys,circle,fill=blue,draw=black,scale=\x1](A42) at (2,2) {};
\node [label={[label distance=\y1 cm]\z1: $0302$},xshift=-\xs,yshift=-\ys,circle,fill=blue,draw=black,scale=\x1](A43) at (\u+1,\u) {};
\node [label={[label distance=\y1 cm]\z1: $0303$},xshift=-\xs,yshift=-\ys,circle,fill=blue,draw=black,scale=\x1](A44) at (2*\u+1,2*\u) {};

\draw [line width=\w1 mm,dashed, opacity=0.2]   (A19)--(A3)--(A33)--(A39) (A20)--(A4)--(A34)--(A40)  (A5)--(A21)  (A6)--(A22) (A23)--(A7)--(A35)--(A41)  (A24)--(A8)--(A36)--(A42)   (A25)--(A9) (A26)--(A10) (A27)--(A11)--(A37)--(A43)  (A28)--(A12)--(A38)--(A44) (A13)--(A29) (A14)--(A30) (A15)--(A31) (A16)--(A32) (A1)--(A2)--(A3)--(A4)--(A1)--(A5)--(A6)--(A7)--(A8)--(A5)--(A15)--(A16)--(A14)--(A13)--(A6) (A2)--(A9)--(A10)--(A12)--(A11)--(A3) (A2)--(A6) (A3)--(A7)  (A4)--(A8)  (A9)--(A11) (A13)--(A15)  (A17)--(A18)--(A19)--(A20)--(A17)--(A21)--(A22)--(A23)--(A24)--(A21)--(A31)--(A32)--(A30)--(A29)--(A22) (A18)--(A25)--(A26)--(A28)--(A27)--(A19) (A18)--(A22) (A19)--(A23)  (A20)--(A24)  (A25)--(A27) (A29)--(A31)  (A33)--(A34)--(A36)--(A35)--(A33)--(A37)--(A38) (A39)--(A40)--(A42)--(A41)--(A39)--(A43)--(A44)   (A26)--(A28) (A32)--(A30) (A16)--(A14)   ;

\draw [line width=\w1 mm] (A41)--(A42)--(A40)--(A39)--(A43)--(A44)  (A38)--(A37)   (A33)--(A34)--(A36)--(A35)  (A7)--(A8)--(A4)--(A3)  (A11)--(A12)--(A10)--(A9)--(A2)--(A1)--(A5)--(A6)--(A13)--(A15)--(A16)--(A14) (A30)--(A32)--(A31)--(A29)--(A22)--(A21)--(A17)--(A18)--(A25)--(A26)--(A28)--(A27)--(A19)--(A20)--(A24)--(A23) ;


 \draw [line width=\w1 mm] (A14)--(A30)  (A7)--(A23) (A38)--(A44) (A35)--(A41); 

\draw [line width=\w1 mm] (A11)--(A37)  (A3)--(A33);

\end{tikzpicture}  
\caption{The Hamiltonian cycle in the Horadam cube $\Pi_4^{2,2}$.} \label{fig:Hamil_cycle_in_Pi_4^{2,2}}
\end{figure}

The next theorem explores the existence of the Hamiltonian cycle in cases of $a$ even and $b$ odd. Setting $b=1$ immediately yields the result for the metallic cubes \cite[Theorem 12]{metallic}. The general construction is very similar; we refer to that paper for more details.

\begin{theorem} Let $a$ be even, $b$ odd and $n>1$ odd. Then the Horadam cube $\Pi^{a,b}_n$ is Hamiltonian. \label{tm:Hor_Ham_cycle_a_even_b_odd}
\end{theorem} 
\begin{proof} Because $a$ is even and $\Pi^{a,b}_n=P_a\square \Pi^{a,b}_{n-1}\oplus P_b\square \Pi^{a,b}_{n-2}$, we can construct Hamiltonian cycle in the subgraph $P_a\square \Pi^{a,b}_{n-1}\oplus P_{b-1}\square \Pi^{a,b}_{n-2}$ in the same manner as in the proof of Theorem \ref{tm: Hor_Ham_cycle_a_even_b_even}. Since $\Pi^{a,b}_1=P_a$, the Hamiltonian cycle in  $P_a\square \Pi^{a,b}_{2}\oplus P_{b-1}\square \Pi^{a,b}_{1}$ can be extended to the cycle in the Horadam cube $\Pi^{a,b}_3$.  Furthermore, $\Pi^{a,b}_2$ is $a\times a$ grid with a path $P_b$ appended to the vertex $0(a-1)$. More precisely, $\Pi^{a,b}_2=P_a^2\oplus P_b$. By the same argument, we can construct a Hamiltonian cycle in $P_a\square \Pi^{a,b}_{3}\oplus P_{b-1}\square \Pi^{a,b}_{2}$ and in $P_a^2\subseteq \Pi^{a,b}_2$. It is easy to verify that those two cycles can be merged into a single cycle in $P_a\square \Pi^{a,b}_{3}\oplus P_{b-1}\square \Pi^{a,b}_{2}\oplus P_a^2$. By Theorem \ref{tm: Horadam_Hamiltonian_path}, the path subgraph induced by vertices that start with $0(a+b-2)$ is part of that cycle. Since $b$ is odd, we can extend the cycle to visit all vertices of the neighboring path induced by vertices that start with $0(a+b-1)$ but one. Thus, we obtained a cycle that visits all vertices but one in $\Pi^{a,b}_4$. 

Now we proceed inductively. Regardless of the parity of $n$, the subgraph $P_a\square \Pi^{a,b}_{n-1}\oplus P_{b-1}\square \Pi^{a,b}_{n-2}$ in $\Pi^{a,b}_{n}$ always contains a Hamiltonian cycle. If $n$ is odd, then $n-2$ is odd too, and the last copy of  $\Pi^{a,b}_{n-2}$ contains a Hamiltonian cycle which can be merged to obtain a cycle in $\Pi^{a,b}_{n}$. If $n$ is even, a similar construction yields a cycle that visits all vertices but one. This completes the proof.  
\end{proof}

In the next theorem, we consider the last remaining case, where $a$ and $b$ are both odd. 

\begin{theorem}
If $a$ and $b$ are odd and $n=3k-1,k\geq 2$, then the Horadam cube $\Pi^{a,b}_n$ is Hamiltonian. \label{tm: Hor_Ham_cycle_a_odd_b_odd}
\end{theorem} \begin{proof}
By Theorem \ref{tm:candec_hor}, $\Pi^{a,b}_5=P_{a-1}\square\Pi^{a,b}_{4}\oplus\Pi^{a,b}_{4} \oplus\Pi^{a,b}_3\oplus P_{b-1}\square \Pi^{a,b}_{3}$. By Theorems \ref{tm: Horadam_Hamiltonian_path} and \ref{tm: Hor_Ham_cycle_a_even_b_even}, whenever the number of copies in the decomposition of the Horadam cube is even, there is a Hamiltonian cycle in that subgraph. Since $a-1$ and $b-1$ are even, there is a cycle that visits all vertices in the subgraphs $P_{a-1}\square\Pi^{a,b}_{4}$ and $P_{b-1}\square\Pi^{a,b}_{3}$. More precisely, there is a cycle in the subgraph $k\Pi^{a,b}_4$ for $1\leq a-1$ and a cycle in the subgraph $0l\Pi^{a,b}_3$ for $a+1\leq l\leq a+b-1$. Moreover, the remaining copies $0\Pi^{a,b}_4$ and $0a\Pi^{a,b}_3$ can be further decomposed as $\Pi^{a,b}_{4}\oplus \Pi^{a,b}_{3}=P_{a+1}\square\Pi^{a,b}_{3}\oplus P_b\square \Pi^{a,b}_{2}$, where $a+1$ is even. Thus we have a Hamiltonian cycle in the subgraph $0(a-1)\Pi^{a,b}_{3}\oplus0a\Pi^{a,b}_{3} \cdots\oplus 0(a+b-1)$. To obtain a Hamiltonian cycle in $\Pi^{a,b}_5$ we just need to verify that a cycle constructed so far can be extended to the subgraph $P_b\square \Pi^{a,b}_{2}$ induced by vertices that start with $00l$ for $a \leq l\leq a+b-1$. Since $\Pi^{a,b}_2$ is $a\times a$ grid with the path on $b$ vertices appended to the vertex $0(a-1)$, the subgraph $P_b\square\Pi^{a,b}_2$ can be decomposed into two grids, $P_{a+b}\square P_{b}$, induced by vertices $00l0k$ for $0\leq k\leq a+b-1$, and $a\leq l\leq a+b-1$, and $P_{a-1}\square P_{a}$. Both grids have an even number of vertices. The construction of the cycle so far is inductive, hence the path subgraph $10a0k$ for $0\leq k\leq a+b-1$ is part of the constructed cycle and it can be extended to encompass the grid $P_{a+b}\square P_{b}$. A similar argument holds for the grid $P_{a-1}\square P_{a}$. Thus, the cycle constructed so far can be extended to the full Hamiltonian cycle in $\Pi^{a,b}_5$. As an example, Figure \ref{fig: Hamil_cycle_in_Pi_5^{3,3}} shows an extension of the cycle to the subgraph $P_3\square \Pi^{3,3}_2\subset \Pi^{3,3}_5$, and it is clear that the extension can be done whenever $a$ and $b$ are odd. The full line denotes the constructed Hamiltonian cycle and the dashed line indicates the extension.

\begin{figure}[h!] \centering \begin{tikzpicture}[scale=0.5]
\tikzmath{\x1 = 0.35; \y1 =-0.05; \z1=90; \w1=0.2; \xs=50; \ys=-30; \yss=-5;
\x2 = \x1 + 1; \y2 =\y1 +3; } 
\scriptsize
\node [label={[label distance=\y1 cm]\z1: $10300$},circle,fill=blue,draw=black,scale=\x1](A1) at (0,4) {};
\node [label={[label distance=\y1 cm]\z1: $10310$},circle,fill=blue,draw=black,scale=\x1](A2) at (0,5.5) {};
\node [label={[label distance=\y1 cm]\z1: $10320$},circle,fill=blue,draw=black,scale=\x1](A3) at (0,7) {};
\node [label={[label distance=\y1 cm]\z1: $10301$},circle,fill=blue,draw=black,scale=\x1](A4) at (-1.5,2.5) {};
\node [label={[label distance=\y1 cm]\z1: $10311$},circle,fill=blue,draw=black,scale=\x1](A5) at (-1.5,4) {};
\node [label={[label distance=\y1 cm]\z1: $10321$},circle,fill=blue,draw=black,scale=\x1](A6) at (-1.5,5.5) {};
\node [label={[label distance=\y1 cm]\z1: $10302$},circle,fill=blue,draw=black,scale=\x1](A7) at (-3,1) {};
\node [label={[label distance=\y1 cm]\z1: $10312$},circle,fill=blue,draw=black,scale=\x1](A8) at (-3,2.5) {};
\node [label={[label distance=\y1 cm]\z1: $10322$},circle,fill=blue,draw=black,scale=\x1](A9) at (-3,4) {};
\node [label={[label distance=\y1 cm]\z1: $10303$},circle,fill=blue,draw=black,scale=\x1](A10) at (-4.5,-0.5) {};
\node [label={[label distance=\y1 cm]\z1: $10304$},circle,fill=blue,draw=black,scale=\x1](A11) at (-6,-2) {};
\node [label={[label distance=\y1 cm]\z1: $10305$},circle,fill=blue,draw=black,scale=\x1](A12) at (-7.5,-3.5) {};

\node [label={[label distance=\y1 cm]\z1: $00300$},xshift=\xs,yshift=\ys,circle,fill=blue,draw=black,scale=\x1](A13) at (A1) {};
\node [label={[label distance=\y1 cm]\z1: $00310$},xshift=\xs,yshift=\ys,circle,fill=blue,draw=black,scale=\x1](A14) at (A2) {};
\node [label={[label distance=\y1 cm]\z1: $00320$},xshift=\xs,yshift=\ys,circle,fill=blue,draw=black,scale=\x1](A15) at (A3) {};
\node [label={[label distance=\y1 cm]\z1: $00301$},xshift=\xs,yshift=\ys,circle,fill=blue,draw=black,scale=\x1](A16) at (A4) {};
\node [label={[label distance=\y1 cm]\z1: $00311$},xshift=\xs,yshift=\ys,circle,fill=blue,draw=black,scale=\x1](A17) at (A5) {};
\node [label={[label distance=\y1 cm]\z1: $00321$},xshift=\xs,yshift=\ys,circle,fill=blue,draw=black,scale=\x1](A18) at (A6) {};
\node [label={[label distance=\y1 cm]\z1: $00302$},xshift=\xs,yshift=\ys,circle,fill=blue,draw=black,scale=\x1](A19) at (A7) {};
\node [label={[label distance=\y1 cm]\z1: $00312$},xshift=\xs,yshift=\ys,circle,fill=blue,draw=black,scale=\x1](A20) at (A8) {};
\node [label={[label distance=\y1 cm]\z1: $00322$},xshift=\xs,yshift=\ys,circle,fill=blue,draw=black,scale=\x1](A21) at (A9) {};
\node [label={[label distance=\y1 cm]\z1: $00303$},xshift=\xs,yshift=\ys,circle,fill=blue,draw=black,scale=\x1](A22) at (A10) {};
\node [label={[label distance=\y1 cm]\z1: $00304$},xshift=\xs,yshift=\ys,circle,fill=blue,draw=black,scale=\x1](A23) at (A11) {};
\node [label={[label distance=\y1 cm]\z1: $00305$},xshift=\xs,yshift=\ys,circle,fill=blue,draw=black,scale=\x1](A24) at (A12) {};

\node [label={[label distance=\y1 cm]\z1: $00400$},xshift=2*\xs,yshift=2*\ys,circle,fill=blue,draw=black,scale=\x1](A25) at (A1) {};
\node [label={[label distance=\y1 cm]\z1: $00410$},xshift=2*\xs,yshift=2*\ys,circle,fill=blue,draw=black,scale=\x1](A26) at (A2) {};
\node [label={[label distance=\y1 cm]\z1: $00420$},xshift=2*\xs,yshift=2*\ys,circle,fill=blue,draw=black,scale=\x1](A27) at (A3) {};
\node [label={[label distance=\y1 cm]\z1: $00401$},xshift=2*\xs,yshift=2*\ys,circle,fill=blue,draw=black,scale=\x1](A28) at (A4) {};
\node [label={[label distance=\y1 cm]\z1: $00411$},xshift=2*\xs,yshift=2*\ys,circle,fill=blue,draw=black,scale=\x1](A29) at (A5) {};
\node [label={[label distance=\y1 cm]\z1: $00421$},xshift=2*\xs,yshift=2*\ys,circle,fill=blue,draw=black,scale=\x1](A30) at (A6) {};
\node [label={[label distance=\y1 cm]\z1: $00402$},xshift=2*\xs,yshift=2*\ys,circle,fill=blue,draw=black,scale=\x1](A31) at (A7) {};
\node [label={[label distance=\y1 cm]\z1: $00412$},xshift=2*\xs,yshift=2*\ys,circle,fill=blue,draw=black,scale=\x1](A32) at (A8) {};
\node [label={[label distance=\y1 cm]\z1: $00422$},xshift=2*\xs,yshift=2*\ys,circle,fill=blue,draw=black,scale=\x1](A33) at (A9) {};
\node [label={[label distance=\y1 cm]\z1: $00403$},xshift=2*\xs,yshift=2*\ys,circle,fill=blue,draw=black,scale=\x1](A34) at (A10) {};
\node [label={[label distance=\y1 cm]\z1: $00404$},xshift=2*\xs,yshift=2*\ys,circle,fill=blue,draw=black,scale=\x1](A35) at (A11) {};
\node [label={[label distance=\y1 cm]\z1: $00405$},xshift=2*\xs,yshift=2*\ys,circle,fill=blue,draw=black,scale=\x1](A36) at (A12) {};

\node [label={[label distance=\y1 cm]\z1: $00500$},xshift=3*\xs,yshift=3*\ys,circle,fill=blue,draw=black,scale=\x1](A37) at (A1) {};
\node [label={[label distance=\y1 cm]\z1: $00510$},xshift=3*\xs,yshift=3*\ys,circle,fill=blue,draw=black,scale=\x1](A38) at (A2) {};
\node [label={[label distance=\y1 cm]\z1: $00520$},xshift=3*\xs,yshift=3*\ys,circle,fill=blue,draw=black,scale=\x1](A39) at (A3) {};
\node [label={[label distance=\y1 cm]\z1: $00501$},xshift=3*\xs,yshift=3*\ys,circle,fill=blue,draw=black,scale=\x1](A40) at (A4) {};
\node [label={[label distance=\y1 cm]\z1: $00511$},xshift=3*\xs,yshift=3*\ys,circle,fill=blue,draw=black,scale=\x1](A41) at (A5) {};
\node [label={[label distance=\y1 cm]\z1: $00521$},xshift=3*\xs,yshift=3*\ys,circle,fill=blue,draw=black,scale=\x1](A42) at (A6) {};
\node [label={[label distance=\y1 cm]\z1: $00502$},xshift=3*\xs,yshift=3*\ys,circle,fill=blue,draw=black,scale=\x1](A43) at (A7) {};
\node [label={[label distance=\y1 cm]\z1: $00512$},xshift=3*\xs,yshift=3*\ys,circle,fill=blue,draw=black,scale=\x1](A44) at (A8) {};
\node [label={[label distance=\y1 cm]\z1: $00522$},xshift=3*\xs,yshift=3*\ys,circle,fill=blue,draw=black,scale=\x1](A45) at (A9) {};
\node [label={[label distance=\y1 cm]\z1: $00503$},xshift=3*\xs,yshift=3*\ys,circle,fill=blue,draw=black,scale=\x1](A46) at (A10) {};
\node [label={[label distance=\y1 cm]\z1: $00504$},xshift=3*\xs,yshift=3*\ys,circle,fill=blue,draw=black,scale=\x1](A47) at (A11) {};
\node [label={[label distance=\y1 cm]\z1: $00505$},xshift=3*\xs,yshift=3*\ys,circle,fill=blue,draw=black,scale=\x1](A48) at (A12) {};

\draw [line width=\w1 mm] (A12)--(A11)--(A10)--(A7)--(A4)--(A1)--(A2)--(A5)--(A8)--(A9)--(A6)--(A3); 

\draw [line width=\w1 mm,dashed] (A12)--(A48)--(A47)--(A11)  (A10)--(A46)--(A43)--(A7) (A4)--(A40)--(A37)--(A1) (A20)--(A21)--(A9)  (A6)--(A18)--(A15)--(A27)--(A30)--(A33)--(A45)--(A42)--(A39)--(A38) (A20)--(A17)--(A14)--(A26)--(A29)--(A32)--(A44)--(A41)--(A38); 

\end{tikzpicture}\caption{A fragment of the Hamiltonian cycle in the Horadam cube $\Pi_5^{3,3}$.} \label{fig: Hamil_cycle_in_Pi_5^{3,3}}
 \end{figure}

Now let $n=3k-1$ and $k\geq 3$. Then $\Pi^{a,b}_{3k-1}=P_{a-1}\square\Pi^{a,b}_{3k-2}\oplus \Pi^{a,b}_{3k-2}\oplus \Pi^{a,b}_{3k-3}\oplus P_{b-1}\square \Pi^{a,b}_{3k-3}$. Furthermore, $\Pi^{a,b}_{3k-2}\oplus \Pi^{a,b}_{3k-3}=P_{a+1}\Pi^{a,b}_{3k-3}\oplus P_b\Pi^{a,b}_{3k-4}$.  The numbers $a-1$, $b-1$, and $a+1$ are even, and, by induction, there is a Hamiltonian cycle in the subgraph $\Pi^{a,b}_{3k-4}$. Since the construction of the cycles is inductive, they can be merged into a single cycle of the Horadam cube $\Pi^{a,b}_{3k-1}$.   
\end{proof}

As an example, Figure \ref{fig: Hamil_cycle_in_Pi_6^{1,3}} shows a construction of the Hamiltonian cycle in the Horadam cube $\Pi^{1,3}_6$.  

\begin{figure}[h!] \centering 
\begin{tikzpicture}[scale=0.5]
\tikzmath{\x1 = 0.35; \y1 =-0.05; \z1=90; \w1=0.2;   \xs=35; \ys=0;
\x2 = \x1 + 1; \y2 =\y1 +3; } 
\small

\node [label={[label distance=\y1 cm]\z1: $00000$},circle,fill=blue,draw=black,scale=\x1](A3) at (0,6) {};
\node [label={[label distance=\y1 cm]\z1: $00001$},circle,fill=blue,draw=black,scale=\x1](A2) at (1.5,7.5) {};
\node [label={[label distance=\y1 cm]\z1: $00002$},circle,fill=blue,draw=black,scale=\x1](A1) at (3,9) {};
\node [label={[label distance=\y1 cm]\z1: $00003$},circle,fill=blue,draw=black,scale=\x1](A22) at (4.5,10.5) {};
\node [label={[label distance=\y1 cm]\z1: $00010$},circle,fill=blue,draw=black,scale=\x1](A4) at (1.5,4.5) {};
\node [label={[label distance=\y1 cm]\z1: $00020$},circle,fill=blue,draw=black,scale=\x1](A5) at (3,3) {};
\node [label={[label distance=\y1 cm]\z1: $00030$},circle,fill=blue,draw=black,scale=\x1](A40) at (4.5,1.5) {};

\node [label={[label distance=\y1 cm]\z1: $00100$},circle,fill=blue,draw=black,scale=\x1](A6) at (-1.5,7.5) {};
\node [label={[label distance=\y1 cm]\z1: $00101$},circle,fill=blue,draw=black,scale=\x1](A7) at (0,9) {};
\node [label={[label distance=\y1 cm]\z1: $00102$},circle,fill=blue,draw=black,scale=\x1](A8) at (1.5,10.5) {};
\node [label={[label distance=\y1 cm]\z1: $00103$},circle,fill=blue,draw=black,scale=\x1](A23) at (3,12) {};

\node [label={[label distance=\y1 cm]\z1: $00200$},circle,fill=blue,draw=black,scale=\x1](A9) at (-3,9) {};
\node [label={[label distance=\y1 cm]\z1: $00201$},circle,fill=blue,draw=black,scale=\x1](A10) at (-1.5,10.5) {};
\node [label={[label distance=\y1 cm]\z1: $00202$},circle,fill=blue,draw=black,scale=\x1](A11) at (0,12) {};
\node [label={[label distance=\y1 cm]\z1: $00203$},circle,fill=blue,draw=black,scale=\x1](A24) at (1.5,13.5) {};

\node [label={[label distance=\y1 cm]\z1: $01000$},xshift=\xs,yshift=\ys,circle,fill=blue,draw=black,scale=\x1](A14) at (0,6) {};
\node [label={[label distance=\y1 cm]\z1: $01001$},xshift=\xs,yshift=\ys,circle,fill=blue,draw=black,scale=\x1](A13) at (1.5,7.5) {};
\node [label={[label distance=\y1 cm]\z1: $01002$},xshift=\xs,yshift=\ys,circle,fill=blue,draw=black,scale=\x1](A12) at (3,9) {};
\node [label={[label distance=\y1 cm]\z1: $01003$},xshift=\xs,yshift=\ys,circle,fill=blue,draw=black,scale=\x1](A25) at (4.5,10.5) {};

\node [label={[label distance=\y1 cm]\z1: $01010$},xshift=\xs,yshift=\ys,circle,fill=blue,draw=black,scale=\x1](A15) at (1.5,4.5) {};
\node [label={[label distance=\y1 cm]\z1: $01020$},xshift=\xs,yshift=\ys,circle,fill=blue,draw=black,scale=\x1](A16) at (3,3) {};
\node [label={[label distance=\y1 cm]\z1: $01030$},xshift=\xs,yshift=\ys,circle,fill=blue,draw=black,scale=\x1](A26) at (4.5,1.5) {};

\node [label={[label distance=\y1 cm]\z1: $02000$},xshift=2*\xs,yshift=2*\ys,circle,fill=blue,draw=black,scale=\x1](A19) at (0,6) {};
\node [label={[label distance=\y1 cm]\z1: $02001$},xshift=2*\xs,yshift=2*\ys,circle,fill=blue,draw=black,scale=\x1](A18) at (1.5,7.5) {};
\node [label={[label distance=\y1 cm]\z1: $02002$},xshift=2*\xs,yshift=2*\ys,circle,fill=blue,draw=black,scale=\x1](A17) at (3,9) {};
\node [label={[label distance=\y1 cm]\z1: $02003$},xshift=2*\xs,yshift=2*\ys,circle,fill=blue,draw=black,scale=\x1](A27) at (4.5,10.5) {};

\node [label={[label distance=\y1 cm]\z1: $02010$},xshift=2*\xs,yshift=2*\ys,circle,fill=blue,draw=black,scale=\x1](A20) at (1.5,4.5) {};
\node [label={[label distance=\y1 cm]\z1: $02020$},xshift=2*\xs,yshift=2*\ys,circle,fill=blue,draw=black,scale=\x1](A21) at (3,3) {};
\node [label={[label distance=\y1 cm]\z1: $02030$},xshift=2*\xs,yshift=2*\ys,circle,fill=blue,draw=black,scale=\x1](A28) at (4.5,1.5) {};

\node [label={[label distance=\y1 cm]\z1: $03000$},xshift=3*\xs,yshift=3*\ys,circle,fill=blue,draw=black,scale=\x1](A33) at (0,6) {};
\node [label={[label distance=\y1 cm]\z1: $03001$},xshift=3*\xs,yshift=3*\ys,circle,fill=blue,draw=black,scale=\x1](A34) at (1.5,7.5) {};
\node [label={[label distance=\y1 cm]\z1: $03002$},xshift=3*\xs,yshift=3*\ys,circle,fill=blue,draw=black,scale=\x1](A35) at (3,9) {};
\node [label={[label distance=\y1 cm]\z1: $03003$},xshift=3*\xs,yshift=3*\ys,circle,fill=blue,draw=black,scale=\x1](A36) at (4.5,10.5) {};

\node [label={[label distance=\y1 cm]\z1: $03010$},xshift=3*\xs,yshift=3*\ys,circle,fill=blue,draw=black,scale=\x1](A37) at (1.5,4.5) {};
\node [label={[label distance=\y1 cm]\z1: $03020$},xshift=3*\xs,yshift=3*\ys,circle,fill=blue,draw=black,scale=\x1](A38) at (3,3) {};
\node [label={[label distance=\y1 cm]\z1: $03030$},xshift=3*\xs,yshift=3*\ys,circle,fill=blue,draw=black,scale=\x1](A39) at (4.5,1.5) {};

\node [label={[label distance=\y1 cm]\z1: $00300$},circle,fill=blue,draw=black,scale=\x1](A29) at (-4.5,10.5) {};
\node [label={[label distance=\y1 cm]\z1: $00301$},circle,fill=blue,draw=black,scale=\x1](A30) at (-3,12) {};
\node [label={[label distance=\y1 cm]\z1: $00302$},circle,fill=blue,draw=black,scale=\x1](A31) at (-1.5,13.5) {};
\node [label={[label distance=\y1 cm]\z1: $00303$},circle,fill=blue,draw=black,scale=\x1](A32) at (0,15) {};

\draw [line width=\w1 mm] (A6)--(A7)--(A2) (A1)--(A8)--(A23)--(A24)--(A32)--(A31)--(A11)--(A10)--(A30)--(A29)--(A9)--(A6) (A22)--(A1)  (A2)--(A3)--(A4)--(A5)--(A40) (A25)--(A12)  (A13)--(A14)--(A15)--(A16)--(A26) (A27)--(A17)   (A18)--(A19)--(A20)--(A21)--(A28) (A39)--(A38)--(A37)--(A33)--(A34)--(A35)--(A36) (A22)--(A25)  (A40)--(A26)  (A27)--(A36)  (A28)--(A39) (A12)--(A17)  (A13)--(A18) ; 
\draw [line width=\w1 mm,dashed, opacity=0.2]   (A9)--(A10)--(A11)--(A24)  (A3)--(A6)--(A9)--(A29) (A2)--(A7)--(A10)--(A30) (A1)--(A8)--(A11)--(A31)   (A1)--(A12)--(A17)--(A35) (A2)--(A13)--(A18)--(A34) (A3)--(A14)--(A19)--(A33) (A4)--(A15)--(A20)--(A37) (A5)--(A16)--(A21)--(A38)  (A40)--(A26)--(A28)--(A39)--(A38)--(A37)--(A33)--(A34)--(A35)--(A36)--(A27)--(A25)--(A22)--(A23)--(A24)--(A32)--(A31)--(A30)--(A29)--(A9);
\end{tikzpicture}
\caption{The Hamiltonian cycle in the Horadam cube $\Pi_6^{1,3}$.} \label{fig: Hamil_cycle_in_Pi_6^{1,3}}
\end{figure}

In \cite{Hamilton}, authors proved that every Fibonacci cube contains a Hamiltonian path and that the Fibonacci cube is Hamiltonian if the number of vertices is even. It follows that the generalization of the Fibonacci cubes to the Horadam cubes preserves all those properties. Namely, summing the results of Theorems \ref{tm: Horadam_Hamiltonian_path}, \ref{tm: Hor_Ham_cycle_a_even_b_even}, \ref{tm:Hor_Ham_cycle_a_even_b_odd} and \ref{tm: Hor_Ham_cycle_a_odd_b_odd}, we can state all results in one simple corollary:
\begin{corollary} \label{cor: Hamiltonicity}
Every Horadam cube contains a Hamiltonian path. If $n>2$ and the number of vertices in a Horadam cube is even, then the Horadam cube is Hamiltonian.
\end{corollary} \begin{proof} The number of vertices in Horadam cubes is determined by the sequence $s^{a,b}_n$ that satisfies the recurrence $s^{a,b}_{n}=as^{a,b}_{n-1}+bs^{a,b}_{n-2}$ with initial values $s^{a,b}_0=1$, $s^{a,b}_1=a$ and $s^{a,b}_{n}=0$ for $n<0$. It is not hard to see that $s^{a,b}_n$ is even in the following cases: if both $a$ and $b$ are even and $n\geq 1$; if both $a$ 
 and $b$ are odd and $n=3k-1$ for some $k\in\mathbb{N}$; and the last case if $a$ is even, and $b$ and $n$ are odd. In the case where $a$ is odd and $b$ is even, the number $s^{a,b}_n$ is odd for every $n$. All these claims can be easily verified by induction. By Theorems \ref{tm: Hor_Ham_cycle_a_even_b_even}, \ref{tm:Hor_Ham_cycle_a_even_b_odd} and \ref{tm: Hor_Ham_cycle_a_odd_b_odd}, the claim follows.   
\end{proof}

In particular, for $b=1$, Theorem \ref{tm: Hor_Ham_cycle_a_odd_b_odd} extends the results for the Hamiltonicity of the metallic cubes for $a$ odd.

Corollary \ref{cor: Hamiltonicity} only provides the necessary conditions for the Horadam cubes to be Hamiltonian. Some algorithms indicate that the Horadam cubes are not Hamiltonian if the number of vertices is odd, but this is yet to be proved (or disproved). 

\section{Concluding remarks} In this paper, we defined and investigated the Horadam cubes, a new family of graphs that generalizes Fibonacci and metallic cubes. Their name reflects the fact that the number of vertices in this family of graphs satisfies the Horadam's recurrence. We explored their basic structural and enumerative properties. As it turned out, the generalized family of graphs retains many appealing and useful properties of the Fibonacci cubes. In particular, it was shown that they admit recursive decomposition, as well as the decomposition into grids, similar to the metallic cubes. Furthermore, the Horadam cubes are induced subgraphs of hypercubes and bipartite median graphs. All Horadam cubes contain a Hamiltonian path, and a Horadam cube is Hamiltonian if the number of vertices is even. While our results offer insight into the characteristics of the new graph family, they fall short of presenting a complete overview. For example, in the 
context of the edges, one can investigate edge general position set \cite{Edge_General_Position_Sets_in_Fibonacci_and_Lucas_Cubes}, and in the context of the degree distribution, our results pave the way for the computation of some degree-based topological invariants. Similarly, our results on distances provide a framework for the derivation of distance-based invariants, including notable examples such as the Wiener and Mostar indices \cite{Doslic-JMC-19}, computed for the Fibonacci and Lucas cubes \cite{MosFibLuc, klavmoll}. Extending our results on the distribution of degrees, a thorough exploration of irregularities, modeled after the calculations for the Fibonacci and Lucas cubes \cite{irreg1, irreg2}, would contribute to a more comprehensive portrait of the Horadam cubes. Another direction of research would include an investigation of the existence of the perfect codes in the Horadam cubes, similar to what was done for the Lucas cubes \cite{mollar1}, or determining the smallest dimension of a hypercube that contains a Horadam cube of dimension $n$ as a subgraph. 


\bibliography{main.bib}
\bibliographystyle{plain}
\end{document}